\documentclass[12pt,draft]{amsart}

\usepackage{a4}
\usepackage[latin1]{inputenc}
\usepackage[T1]{fontenc}

\begin{document}

\title{\bf On the classification of rational surface singularities}
\author{Jan Stevens}
\address{Matematiska vetenskaper, G\"oteborgs universitet 
och Chalmers tekniska h\"og\-skola, 41296
G\"oteborg, Sweden}
\email{stevens@chalmers.se}

\def\Bbb{\mathbb} \def\cal{\mathcal}
\def\wt#1{\widetilde{#1}}
\def\sier#1{{\cal O}_{#1}}
\def\C{{\Bbb C}} \def\Q{{\Bbb Q}}
\def\Z{{\Bbb Z}} \def\P{{\Bbb P}}
\def\cE{{\cal E}} \def\cW{{\cal W}}
\def\roep #1.{\medbreak\noindent{\bf #1\/}.\enspace}
\def\endroep{\par\medbreak}
\long\def\comment#1\endcomment{}

\newtheorem{theorem}{Theorem}[section]
\newtheorem{prop}[theorem]{Proposition}
\newtheorem{lemma}[theorem]{Lemma}
\newtheorem{cor}[theorem]{Corollary}
\newtheorem*{conj}{Conjecture}
\newtheorem*{question}{Question}
\newtheorem*{propstar}{Proposition}
\newtheorem*{classification}{Classification}
\theoremstyle{definition}
\newtheorem{defn}[theorem]{Definition}
\newtheorem{remark}[theorem]{Remark}
\newtheorem{example}[theorem]{Example}

\newcommand{\bp}{\begin{picture}}
\newcommand{\bpn}{\begin{picture}(0,0)}
\newcommand{\ep}{\end{picture}}

\def\vir{\makebox(0,0){\rule{8pt}{8pt}}}
\def\cir{\circle*{0.23}}

\unitlength 30pt

\def\ci#1{%
{\bpn\put(0,0){\cir}\put(0,0.2){\makebox(0,0)[b]{$\scriptstyle #1$}}\ep}}
\def\vi#1{%
{\bpn\put(0,0){\vir}\put(0,0.3){\makebox(0,0)[b]{$\scriptstyle #1$}}\ep}}
\def\vis#1{%
{\bpn\put(0,0){\vir}\put(0.1,0.3){\makebox(0,0)[bl]{$\scriptstyle #1$}}\ep}}
\def\vit#1{%
{\bpn\put(0,0){\vir}\put(0.15,0){\makebox(0,0)[l]{$\scriptstyle #1$}}\ep}}
\def\cirr#1{%
{\bpn\put(0,0){\cir}\put(0.2,0){\makebox(0,0)[l]{$\scriptstyle #1$}}\ep}}

\def\lijn{\line(1,0){1}}
\def\keten{{\bpn\put(0,0){\line(1,0){0.5}}
 \put(1,0){\makebox(0,0){$\dots$}}\put(2,0){\line(-1,0){0.5}}\ep}}
\def\ketvert{{\bpn\put(0,0){\line(0,-1){0.5}}
 \put(0,-.7){\makebox(0,0){$\vdots$}}\put(0,-1.6){\line(0,1){0.5}}\ep}}

\def\ketDia{{\bpn\put(0,0){\line(2,1){0.5}}
 \put(1,.5){\makebox(0,0){$\Ddots$}}\put(2,1){\line(-2,-1){0.5}}\ep}}
\def\ketdia{{\bpn\put(0,0){\line(2,-1){0.5}}
 \put(1,-.4){\makebox(0,0){$\jddots$}}\put(2,-1){\line(-2,1){0.5}}\ep}}
\def\ketlDia{{\bpn\put(0,0){\line(-2,1){0.5}}
 \put(-1,.6){\makebox(0,0){$\jddots$}}\put(-2,1){\line(2,-1){0.5}}\ep}}
\def\ketldia{{\bpn\put(0,0){\line(-2,-1){0.5}}
 \put(-1,-.5){\makebox(0,0){$\Ddots$}}\put(-2,-1){\line(2,1){0.5}}\ep}}
\def\Ddots{\mathinner{\mkern1mu\raise1pt
\vbox{\kern7pt\hbox{.}}\kern3pt
\raise4pt\hbox{.}\kern3pt\raise7pt\hbox{.}\mkern1mu}}
\def\jddots{\mathinner{\mkern1mu\raise7pt
\vbox{\kern7pt\hbox{.}}\kern3pt
\raise4pt\hbox{.}\kern3pt\raise1pt\hbox{.}\mkern1mu}}

\def\nup#1[#2]{\bpn(0,1)\put(0,0){\line(0,1){1}}#1{0}
 \put(-0.2,0){\makebox(0,0)[r]{$\scriptstyle #2$}}\ep}


\begin{abstract}
A general strategy is given for the classification
of graphs of rational surface singularities.
For each maximal rational double point configuration
we investigate the possible multiplicities in the fundamental cycle.
We classify completely certain types of graphs.
This allows  to extend the classification of rational singularities 
to multiplicity 8. 
We also discuss the complexity of rational resolution graphs.
\end{abstract}

\maketitle

\section*{Introduction}
The topological classification of complex surface singularities
amounts to classifying resolution graphs. Such a graph represents a
complex curve on a surface, and the simplest case is when this curve
is rational;
then the singularity is called rational and the graph in fact determines
the analytical type of the singularity up to equisingular deformations.

Classification of singularities tends to lead to long lists,
but making them is not a purpose on its own.
Sometimes one  wants a list to prove statements by case by case
checking. If the lists become too unwieldy, as in the case on hand,
their main use will be to provide an ample supply of examples to
test conjectures on. With this objective the most useful description
of rational resolution graphs is as a list of parts, together with
assembly instructions, guaranteeing that the result is a rational
graph. For a special class of rational singularities, those with
almost reduced fundamental cycle, such a classification exists
\cite{ro, gu}. 

As prototype of our classification and to fix notations
we first treat the special case. The fundamental cycle
(\cite{ar}, see also Definition \ref{deffc}) can be seen
as divisor on the exceptional set of the resolution,
with positive coefficients (and it is this divisor which
should be rational as non-reduced curve). It is characterised
numerically as the minimal positive cycle intersecting
each exceptional curve non-positively, and can therefore be
computed using the intersection form encoded in the graph.
The fundamental cycle  is called almost reduced if it
is reduced at  the non-$(-2)$'s. So higher multiplicities 
can only occur on the maximal 
rational double point (RDP) configurations. 
The classification splits in two parts: one has to 
determine the multiplicities on the RDP-configurations and
how they can be attached to the rest of the graph. The
explicit list of graphs  can be found in the 
paper by Gustavsen \cite{gu}. Blowing down the 
RDP-configurations to rational double point singularities
gives the canonical model or RDP-resolution.  Its 
exceptional set can again by described by a graph.
Our classification strategy in general  is
to first find the graphs for the RDP-resolution,
and then determine
which rational double point (RDP) configurations can occur.

The first results in this paper are on
graphs, where each RDP-configuration is attached to at most 
one  non-reduced non-$(-2)$. 
The possible graphs for the RDP-resolution are easy to describe,
but here a new phenomenon occurs, that not every
candidate graph can be realised by a rational singularity. In particular,
if the graph contains only one  non-$(-2)$, this vertex
has multiplicity at most 6 in the fundamental cycle.
These considerations apply to all multiplicities, but only for a
restricted class of singularities; they cover all singularities of
low multiplicity.
Our present results extend the classification
of rational singularities of multiplicity 4 \cite{st-p}, 
and  allow to recover the classification by Tosun et al.~for
multiplicity 5 \cite{tosun}.

Multiplicity 6 necessitates the study of RDP-configurations, connecting
two non-reduced non-$(-2)$'s.
We first determine the conditions under which the multiplicities in
the fundamental cycle become as high as possible. We do this
for each  RDP-configuration separately. The existence depends
on the rest of the graph.  Then we use the
same computations to treat the case that the  non-$(-2)$'s have
multiplicity exactly two. This allows us to complete the classification
of  rational singularities of multiplicity 6. 
The same methods work for multiplicity 7 and 8,
but we do not treat these cases explicitly,  except for one new case,
of three  non-reduced non-$(-2)$'s, with which we  conclude our classification.

We do not claim that it is feasible to treat all
multiplicities with our methods.
Our last result, on multiplicity 8,
gives a glimpse of what is needed in general. To use
induction over the number of non-$(-2)$'s, one needs
detailed knowledge on the graphs for lower multiplicity, and it
does not suffice to compute with  RDP-configurations separately. We include (at the end of the
first section) a non-trivial example  of a rational graph,
of multiplicity 37; the graph of the canonical model is 
rather simple. This example comes from a paper by
Karras \cite{ka}, which maybe contains 
the deepest study of the structure of 
resolution graphs in the literature.
He proves that every rational singularity
deforms  into a
cone over a rational normal curve of the same multiplicity.
My main motivation for taking up the classification again lies in
the same  direction. The ultimate
goal is to study the Artin component of the
semi-universal deformation. Over this component a simultaneous
resolution exists (or, without base change, a simultaneous
canonical model). 
This is one motivation of  our
classification strategy of first  finding 
the graph for the RDP-resolutions.
The analytical type of 
the total space over the Artin component
(up to smooth factors) is an interesting
invariant of the singularity.
In his thesis \cite{ro} Ancus Röhr turned the
problem of formats around and defined the format as 
just this invariant.  He showed that the
format determines the exceptional set of the canonical model
of the singularity. Examples in this paper  cast doubt on our
earlier conjecture that the converse holds.

RDP-configurations can be of type $A$, $D$ and $E$. Our computations show
that one cannot reach high multiplicities in the fundamental cycle
using configurations of type $D$ and $E$. 
With this goal it suffices to look at configurations
of type $A$. Indeed, the picture which arises from our
classifications, is that for most purposes it suffices to look
at rather simple configurations of type $A$.

One answer to the question how complex
a graph can be is that of Lê and Tosun \cite{lt}, who take the number
of rupture points (vertices with valency at least 3) as measure.
We give a simplified proof of their estimate, that this number is
bounded by $m-2$, where $m$ is the multiplicity of the singularity.
Our argument shows that the highest complexity is attained by
graphs with reduced fundamental cycle.

The structure of this paper is as follows. In the first section
we review some properties of resolution graphs. The next section gives
the classification of singularities with almost reduced fundamental cycle.
Section 3 is about  complexity in the sense of \cite{lt}. Then we
discuss the format of a rational singularity, following \cite{ro}.
Our computations use a special way to compute the fundamental cycle,
which we explain in Section 5. The case, where  
each RDP-configuration is attached to at most one 
non-reduced non-$(-2)$, is treated in Section 6, while the following
section describes RDP-configurations on general graphs. In the
final section we complete the classification for multiplicity 6
and treat  the case of three  non-reduced non-$(-2)$'s
in multiplicity 8.

\section{Rational graphs}\label{sect_een}
In this section we review some properties of resolution graphs.
References are Artin \cite{ar}, Wagreich \cite{wag} and Wall
\cite{wal}, and for rational singularities in addition Laufer
\cite{la}.

The topological type of a normal complex
surface singularity is determined by and  determines
the resolution graph of the minimal good resolution \cite{ne}.
A resolution graph 
can be defined for any resolution. 

\begin{defn}
Let $\pi\colon (M,E)\to (X,p)$ be a resolution of a surface 
singularity with exceptional divisor $E=\bigcup_{i=1}^r E_i$.
The \textit{resolution graph}  $\Gamma$ is a weighted graph 
with vertices 
corresponding to the irreducible components $E_i$. 
Each vertex has two weights,
the self-intersection $-b_i=E_i^2$, and the arithmetic genus $p_a(E_i)$, the second
traditionally written in square brackets and omitted if zero.
There is an edge between vertices 
if the corresponding components $E_i$ and $E_j$ intersect,  
weighted with the intersection number
$E_i\cdot E_j$  (only written out if larger than one).
\end{defn}
Other definitions, which record more information, are possible:
one variant is to have an edge 
for 
each intersection point $P\in E_i\cap E_j$, with weight
the local intersection number $(E_i\cdot E_j)_P$. These 
subtleties need not concern us here, as the exceptional divisor
of a rational singularity is a simple normal crossings divisor.

We call the vertices of the graph $\Gamma$ also for $E_i$. This
should  cause no confusion. From the context it will be clear whether we
consider $E_i$ as vertex or as curve. In fact, we use $E_i$ also in
a third sense. The classes of the curves $E_i$
form a preferred basis of $H:=H_2(M,\Z)$. 
Following
algebro-geometric tradition the elements of $H$ are called
\textit{cycles}. They are written as linear combinations of the $E_i$.

The resolution graph (as defined above)
is also the graph of the quadratic lattice $H:=H_2(M,\Z)$, in the sense of
\cite{lw}.  The intersection form on $M$ gives  a negative definite
quadratic form on $H$. Let $K\in H^2(M,\Z)=H^\#$ be the canonical
class. It can be written as rational cycle in $H_\Q=H\otimes \Q$ by solving
the adjunction equations $E_i\cdot(E_i+K)=2p_a(E_i)-2$.
The function $-\chi(A)=\frac12 A\cdot(A+K)$, $A\in H$, makes $H$ 
into a \textit{quadratic lattice} \cite[1.4]{lw}. 
We prefer to work with the genus
$p_a(A)=1-\chi(A)$. Note that the genus function determines the
intersection form, as
\[
p_a(A+B)=p_a(A)+p_a(B)+A\cdot B -1\;.
\]
The data $(H,p_a)$ is equivalent to $(H, \{E_i\cdot E_j\},
\{p_a(E_i)\})$, encoded in the resolution graph $\Gamma$.
Sometimes we identify $H$ with the free abelian group on the vertex
set of $\Gamma$, and talk about cycles on $\Gamma$.

\begin{defn}
A cycle $A=\sum a_iE_i$ (in $H$ or $H_\Q$) is \textit{effective}
 or non-negative,
$A\geq 0$, if all $a_i\geq0$. 
There is a natural inclusion $j\colon
H\to H^\#$, given by $j(A)(B)=-A\cdot B$ (note the minus sign,
because of negative definiteness). 
A cycle $A$ is \textit{anti-nef}, if
$j(A)\geq 0$ in $H^\#$, i.e., $A\cdot E_i\leq0$ for all $i$.  The
anti-nef elements in $H$ form a semigroup $\cE$ and one writes
$\cE^+$ for $\cE\setminus\{0\}$.
\end{defn}
If $A$ is anti-nef, then
$A\geq0$. Indeed, write $A=A_+-A_-$ with $A_+$, $A_-$ non-negative
cycles with no components in common. Then $0\leq -A\cdot A_-=
A_-^2-A_+\cdot A_- \leq A_-^2$, so by negative definiteness $A_-=0$.
Furthermore, if  $A\in \cE^+$, then $A\geq E$, where $E=\sum E_i$ is
the reduced exceptional cycle. Indeed, if the support of $A$ is not
the whole of $E$, then there exists an $E_i$ intersecting $A$ strict
positively, as $A>0$, and $E$ is connected.

\begin{defn}
Given two cycles $A=\sum a_iE_i$, $B=\sum b_iE_i$, their \textit{infimum} is
the cycle $\inf(A,B)=\sum c_iE_i$ with $c_i=\min
(a_i,b_i)$ for all $i$. This definition extends to subsets of
$\cE$.
\end{defn}

\begin{lemma}
Let  $\cW\subset \cE^+$ be a subset. Then $\inf\; \cW\in \cE^+$.
\end{lemma}

\begin{proof}
Let $W=\inf \cW$. Fix an $i$ and choose $A\in \cW$ with $a_i$
minimal. Then $0\geq E_i\cdot A= E_i\cdot(A-W)+E_i\cdot W\geq
E_i\cdot W$, as $A-W\geq0$ with coefficient 0 at $E_i$. So $W\cdot
E_i\leq 0$ for all $i$.  As $A\geq E$ for all $A\in \cW$, also $W\geq E>0$. 
\end{proof}

\begin{defn}\label{deffc}
The \textit{fundamental cycle} $Z$ is the cycle $\inf \cE^+$.
\end{defn}

In other words, the cycle $Z$ is the smallest cycle such that
$E_i\cdot Z\leq0$ for all $i$. It can be computed with a
\textit{computation sequence} \cite{la}. Start with any cycle $Z_0$
known to satisfy $Z_0\leq Z$; one such cycle is $E$. Let $Z_k$ be
computed. If $Z_k\neq Z$, then there is an $E_{j(k)}$ with
$Z_k\cdot E_{j(k)}>0$. Define $Z_{k+1}= Z_k+E_{j(k)}$. Then $(Z -
Z_k)\cdot E_{j(k)}<0$, so $E_{j(k)}$ lies in the support of $Z -
Z_k$, giving $E_{j(k)}\leq Z-Z_k$. Therefore $Z_{k+1}\leq Z$.

The fundamental cycle depends of course on the chosen resolution,
but in an easily controlled way. Therefore it can be used to
define invariants of the singularity  \cite{wag}.

Let $\sigma\colon M'\to M$ be
the blow-up in a point of $E$, with exceptional divisor $E_0'$. The
exceptional divisor of $M'\to X$ is $E'=E_0'+\sum_{i=1}^r E_i'$,
where the $E_i'$, $i\geq1$ are mapped onto the $E_i$. For a 
cycle $A=\sum a_iE_i$ on $M$  the pull-back $\sigma^*A$ is defined
as
\[
\sigma^*A=a_0E_0'+A^\#\;, \quad \textrm{ where }A^\#=\sum _{i=1}^r
a_i E_i' \textrm{ and } E_0'\cdot \sigma^*A=0\;.
\]
In fact, $a_0$ is the multiplicity of $A$ in the point blown up.
The main property of the intersection product in this connection is
that $\sigma^*A \cdot \sigma^*B =A\cdot B$. This product is then
also equal to $\sigma^*A \cdot B^\#$. 

The canonical cycle on $M'$
satisfies $K'=\sigma^*K+E_0'$. This gives that $\sigma^*A\cdot
K'=\sigma^*A\cdot(\sigma^*K+E_0')= \sigma^*A\cdot \sigma^*K=A\cdot
K$ and therefore $p_a(\sigma^*A) =p_a(A)$.

\begin{lemma} The fundamental cycle $Z'$ on $M'$ is $\sigma^*Z$, the
pull back of the fundamental cycle on $M$.
\end{lemma}

\begin{proof} One has   $E_0'\cdot Z'=0$, for otherwise $Z'-E_0'$ is anti-nef.
Therefore $Z'= \sigma^*Y$ for some cycle $Y$ and $Y\cdot E_i=
\sigma^*Y \cdot \sigma^*E_i=Z'\cdot E_i'\leq 0$, so $Z\leq Y$. On the other
hand, $\sigma^*Z\in \cE'$, so $\sigma^*Y=Z'\leq\sigma^*Z$.
\end{proof}

\begin{cor} The genus $p_a(Z)$ and degree $-Z^2$ of the fundamental cycle
are invariants of the singularity.
\end{cor}

\begin{defn} The \textit{fundamental genus} of a singularity is the genus $p_a(Z)$
of the fundamental cycle.
\end{defn}

A singularity has also an arithmetic genus \cite{wag} (the largest value
of $p_a(D)$ over all effective cycles $D$), but this is a less interesting 
invariant. More important is the geometric genus, which is
$h^1(\sier M)$, and also the  largest value of  $h^1(\sier D)$
over all effective cycles $D$. 

Rational singularities were introduced by  Artin \cite{ar} using the
geometric genus of singularities. He proved the following
characterisation, which we take as definition.

\begin{defn}  A normal surface singularity is \textit{rational} if
its fundamental genus $p_a(Z)$ is equal to $0$.
\end{defn}

Artin also proves that  degree $-Z^2$ of the fundamental cycle 
is equal to the multiplicity
$m$ of the singularity. The embedding dimension of $X$ is $m+1$,
which is maximal for normal surface singularities of multiplicity
$m$.

\begin{theorem}[Laufer's rationality criterion]
A resolution graph represents a rational singularity if and only if
\begin{itemize}
\item each vertex $E_i$  has $p_a(E_i)=0$,
\item if a cycle $Z_k$ occurs in a computation sequence  
and if $Z_k\cdot E_i>0$, then $Z_k\cdot E_i=1$.
\end{itemize}
\end{theorem}
For the \lq if\rq-direction it suffices to have the second property
for the steps in one computation sequence,
starting from a single vertex.
The criterion follows from the fact that the genus cannot decrease
in a computation sequence, as 
$p_a(Z_k+E_i)=p_a(Z_k)+p_a(E_i)+Z_k\cdot E_i -1$. 

All irreducible components of the exceptional set have to be smooth
rational curves,  pairwise  intersecting transversally in at most one point. 
This shows that minimal resolution  of a rational singularity
is a good resolution.

Following 
Lê--Tosun \cite{lt}
we call the minimal resolution graph of a rational singularity a
\textit{rational graph}. 
It can be characterised combinatorically as
weighted tree (with only vertex weights $-b_i\leq -2$), representing
a negative definite quadratic form, such that the genus of the
fundamental cycle is 0.

The  main invariant of a rational graph 
is its \textit{degree} $-Z^2$.
It is related to the \textit{canonical degree} $Z\cdot K$
by $-Z^2=Z\cdot K+2$, as  $p_a(Z)=0$. Let $Z=\sum z_iE_i$, $-b_i=E_i^2$. Then
\[
Z\cdot K = \sum z_i(b_i-2)\;.
\]
So the degree is determined by the coefficients $z_i$
of the fundamental cycle at non-$(-2)$-vertices $E_i$.

As example of a rational graph we show the one
(of degree 37) occurring in the paper of Karras \cite{ka}. Every
\rule{8pt}{8pt}
is a $(-3)$-vertex. The numbers are the coefficients of the fundamental
cycle.
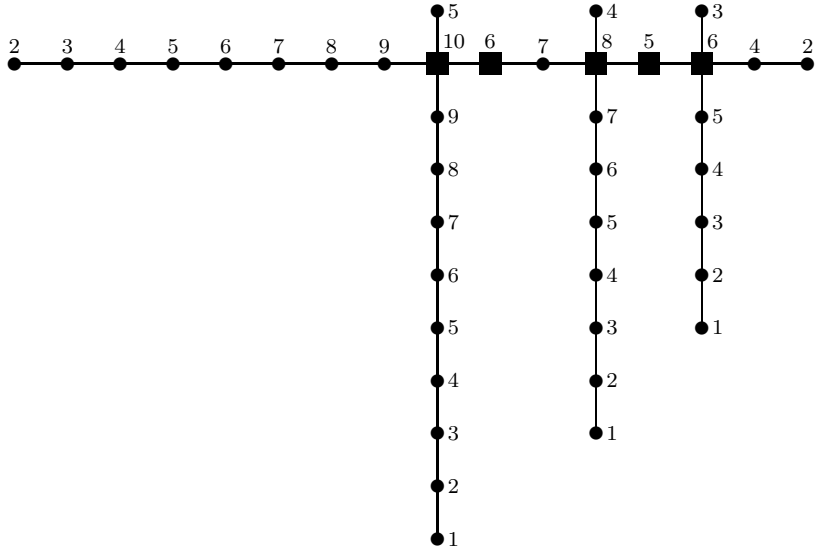
\[
\unitlength=20pt
\bp(15,10)(1,-9)
\put(1,0){\ci2} \put(1,0){\lijn}
\put(2,0){\ci3} \put(2,0){\lijn}
\put(3,0){\ci4} \put(3,0){\lijn}
\put(4,0){\ci5} \put(4,0){\lijn}
\put(5,0){\ci6} \put(5,0){\lijn}
\put(6,0){\ci7} \put(6,0){\lijn}
\put(7,0){\ci8} \put(7,0){\lijn}
\put(8,0){\ci{9}} \put(8,0){\lijn}
\put(9,0){\vis{10}} \put(9,0){\lijn}
\put(9,0){\line(0,1){1}}  \put(9,1){\cirr 5}
\put(9,0){\line(0,-1){1}}  \put(9,-1){\cirr 9}
\put(9,-1){\line(0,-1){1}}  \put(9,-2){\cirr 8}
\put(9,-2){\line(0,-1){1}}  \put(9,-3){\cirr 7}
\put(9,-3){\line(0,-1){1}}  \put(9,-4){\cirr 6}
\put(9,-4){\line(0,-1){1}}  \put(9,-5){\cirr 5}
\put(9,-5){\line(0,-1){1}}  \put(9,-6){\cirr 4}
\put(9,-6){\line(0,-1){1}}  \put(9,-7){\cirr 3}
\put(9,-7){\line(0,-1){1}}  \put(9,-8){\cirr 2}
\put(9,-8){\line(0,-1){1}}  \put(9,-9){\cirr 1}
\put(10,0){\vi{6}} \put(10,0){\lijn}
\put(11,0){\ci7} \put(11,0){\lijn}
\put(12,0){\vis8} \put(12,0){\lijn}
\put(12,0){\line(0,1){1}}  \put(12,1){\cirr 4}
\put(12,0){\line(0,-1){1}}  \put(12,-1){\cirr 7}
\put(12,-1){\line(0,-1){1}}  \put(12,-2){\cirr 6}
\put(12,-2){\line(0,-1){1}}  \put(12,-3){\cirr 5}
\put(12,-3){\line(0,-1){1}}  \put(12,-4){\cirr 4}
\put(12,-4){\line(0,-1){1}}  \put(12,-5){\cirr 3}
\put(12,-5){\line(0,-1){1}}  \put(12,-6){\cirr 2}
\put(12,-6){\line(0,-1){1}}  \put(12,-7){\cirr 1}
\put(13,0){\vi5} \put(13,0){\lijn}
\put(14,0){\vis6} \put(14,0){\lijn}
\put(14,0){\line(0,1){1}}  \put(14,1){\cirr 3}
\put(14,0){\line(0,-1){1}}  \put(14,-1){\cirr 5}
\put(14,-1){\line(0,-1){1}}  \put(14,-2){\cirr 4}
\put(14,-2){\line(0,-1){1}}  \put(14,-3){\cirr 3}
\put(14,-3){\line(0,-1){1}}  \put(14,-4){\cirr 2}
\put(14,-4){\line(0,-1){1}}  \put(14,-5){\cirr 1}
\put(15,0){\ci4} \put(15,0){\lijn}
\put(16,0){\ci2}
\ep
\]

\section{Almost reduced fundamental cycle}
\label{arfc}
As the lists in the classification become unwieldy, we first treat
a simple special case, where only $(-2)$ vertices can have higher
multiplicity in the fundamental cycle. 
Its classification is contained in the thesis
of R\"ohr \cite{ro} as part of more general results.
The explicit list (Tables \ref{tableA}, \ref{tableB} 
and \ref{tableC}) 
of graphs of RDP-configurations
can be found with Gustavsen \cite{gu}.

\begin{defn}[\cite{la3}] A rational singularity has an \textit{almost reduced}
fundamental cycle if the fundamental cycle $Z=\sum z_iE_i$ 
on the minimal resolution  is reduced at the
non-$(-2)$'s, i.e., $z_i=1$ if $b_i>2$. 
\end{defn}

We also talk
about rational graphs with almost reduced fundamental cycle.

One can compute the fundamental cycle
starting from the reduced exceptional cycle by only adding curves
occurring in rational double  point configurations. The computation
can be done for each configuration separately. Therefore we
start with these configurations.

\begin{table}\caption{RDP-configurations, attached to one curve}\label{tableA}
\begin{list}{}{}
\item[$A_n^k${\rm :}] 
\[
\bp(8,1)(0,-1)
\put(0,0){\ci1}  \put(0,0){\keten}
\put(2,0){\ci{k-1}} \put(2,0){\lijn}
\put(3,0){\ci k}  \put(3,0){\keten}
\put(3,0){\line(0,-1){1}} \put(3,-1){\vir}
\put(5,0){\ci {k}}
\put(5,-.11){\makebox(0,0)[t]{$\downarrow$}} \put(5,0){\lijn}
\put(6,0){\ci {k-1}}  \put(6,0){\keten}
\put(8,0){\ci1}
\ep
\]
\item[${}^{\it I}\!D_k^2${\rm :}] 
\[
\bp(4,1)(0,-1)
\put(0,0){\vir}  \put(0,0){\lijn}
\put(1,0){\ci2} \put(1,0){\keten}
\put(1,-.11){\makebox(0,0)[t]{$\downarrow$}}
\put(3,0){\ci 2}  \put(3,0){\lijn}
\put(3,0){\line(0,-1){1}} \put(3,-1){\cirr1}
\put(4,0){\ci 1}
\ep
\]

\item[${}^{\it II}\!D_{2k}^{k}${\rm :}] 
\[
\bp(5,1)(0,-1)
\put(0,0){\vir}  \put(0,0){\lijn}
\put(1,0){\ci k} \put(1,0){\lijn}
\put(1,-.11){\makebox(0,0)[t]{$\downarrow$}}
\put(2,0){\ci {2k-2}}  \put(2,0){\lijn}
\put(3,0){\ci {2k-3}}  \put(3,0){\keten}
\put(2,0){\line(0,-1){1}} \put(2,-1){\cirr {k-1}}
\put(5,0){\ci 1}
\ep
\]
\item[${}^{\it II}\!D_{2k+1}^{k}${\rm :}] 
\[
\bp(5,1)(0,-1)
\put(0,0){\vir}  \put(0,0){\lijn}
\put(1,0){\ci k} \put(1,0){\lijn}
\put(2,0){\ci {2k-1}}  \put(2,0){\lijn}
\put(3,0){\ci {2k-2}}  \put(3,0){\keten}
\put(2,0){\line(0,-1){1}} \put(2,-1){\cirr k}
\put(1.91,-1){\makebox(0,0)[r]{$\leftarrow$}}
\put(5,0){\ci 1}
\ep
\]

\item[$E_6^2${\rm :}] 
\[
\bp(5,1)(0,-1)
\put(0,0){\vir}  \put(0,0){\lijn}
\put(1,0){\ci 2} \put(1,0){\lijn}
\put(2,0){\ci 3}  \put(2,0){\lijn}
\put(3,0){\ci 4}  \put(3,0){\lijn}
\put(3,0){\line(0,-1){1}} \put(3,-1){\cirr 2}
\put(4,0){\ci 3}  \put(4,0){\lijn}
\put(5,0){\ci 2}
\put(5,-.11){\makebox(0,0)[t]{$\downarrow$}}
\ep
\]
\item[$E_7^3${\rm :}] 
\[
\bp(6,1)(0,-1)
\put(0,0){\vir}  \put(0,0){\lijn}
\put(1,0){\ci 3} \put(1,0){\lijn}
\put(2,0){\ci 4}  \put(2,0){\lijn}
\put(3,0){\ci 5}  \put(3,0){\lijn}
\put(4,0){\ci 6}  \put(4,0){\lijn}
\put(4,0){\line(0,-1){1}} \put(4,-1){\cirr 3}
\put(5,0){\ci 4}  \put(5,0){\lijn}
\put(6,0){\ci 2}
\put(1,-.11){\makebox(0,0)[t]{$\downarrow$}}
\ep
\]
\end{list}
\end{table}
\begin{table}\caption{RDP-configurations, attached to two curves}\label{tableB}
\begin{list}{}{}
\item[$A_n^{1,1}${\rm :}] 
\[
\bp(4,1)(0,-1)
\put(0,0){\vir}  \put(0,0){\lijn}
\put(1,0){\ci1}  \put(1,0){\keten}
\put(3,0){\ci1} \put(3,0){\lijn}
\put(4,0){\vir}
\ep
\]
\item[${}^{\it I}\!A_n^{2,k}${\rm :}] 
\[
\bp(9,1)(-1,-1)
\put(-1,0){\vir}  \put(-1,0){\lijn}
\put(0,0){\ci2}  \put(0,0){\keten}
\put(2,0){\ci{k-1}} \put(2,0){\lijn}
\put(3,0){\ci k}  \put(3,0){\keten}
\put(3,0){\line(0,-1){1}} \put(3,-1){\vir}
\put(5,0){\ci {k}}
\put(5,-.11){\makebox(0,0)[t]{$\downarrow$}} \put(5,0){\lijn}
\put(6,0){\ci {k-1}}  \put(6,0){\keten}
\put(8,0){\ci1}
\ep
\]
\item[${}^{\it II}\!A_n^{k,2}${\rm :}] 
\[
\bp(9,1)(0,-1)
\put(0,0){\ci1}  \put(0,0){\keten}
\put(2,0){\ci{k-1}} \put(2,0){\lijn}
\put(3,0){\ci k}  \put(3,0){\keten}
\put(3,0){\line(0,-1){1}} \put(3,-1){\vir}
\put(5,0){\ci {k}}
\put(5,-.11){\makebox(0,0)[t]{$\downarrow$}} \put(5,0){\lijn}
\put(6,0){\ci {k-1}}  \put(6,0){\keten}
\put(8,0){\ci2} \put(8,0){\lijn}
\put(9,0){\vir}
\ep
\]

\item[$D_{2k+1}^{k+1,2}${\rm :}] 
\[
\bp(6,1)(0,-1)
\put(0,0){\vir}  \put(0,0){\lijn}
\put(1,0){\ci {k+1}} \put(1,0){\lijn}
\put(1,-.11){\makebox(0,0)[t]{$\downarrow$}}
\put(2,0){\ci {2k}}  \put(2,0){\lijn}
\put(3,0){\ci {2k-1}}  \put(3,0){\keten}
\put(2,0){\line(0,-1){1}} \put(2,-1){\cirr {k}}
\put(5,0){\ci 2} \put(5,0){\lijn}
\put(6,0){\vir}
\ep
\]
\item[$D_{2k}^{k,2}${\rm :}] 
\[
\bp(6,1)(0,-1)
\put(0,0){\vir}  \put(0,0){\lijn}
\put(1,0){\ci k} \put(1,0){\lijn}
\put(2,0){\ci {2k-1}}  \put(2,0){\lijn} 
\put(3,0){\ci {2k-2}}  \put(3,0){\keten}
\put(2,0){\line(0,-1){1}} \put(2,-1){\cirr k}
\put(1.91,-1){\makebox(0,0)[r]{$\leftarrow$}}
\put(5,0){\ci 2} \put(5,0){\lijn}
\put(6,0){\vir}
\ep
\]
\end{list}
\end{table}
\begin{table}\caption{RDP-configurations, attached to three curves}\label{tableC}
\begin{list}{}{}
\item[$A_n^{2,k,2}${\rm :}] 
\[
\bp(10,1)(-1,-1)
\put(-1,0){\vir}  \put(-1,0){\lijn}
\put(0,0){\ci2}  \put(0,0){\keten}
\put(2,0){\ci{k-1}} \put(2,0){\lijn}
\put(3,0){\ci k}  \put(3,0){\keten}
\put(3,0){\line(0,-1){1}} \put(3,-1){\vir}
\put(5,0){\ci {k}}
\put(5,-.11){\makebox(0,0)[t]{$\downarrow$}} \put(5,0){\lijn}
\put(6,0){\ci {k-1}}  \put(6,0){\keten}
\put(8,0){\ci2} \put(8,0){\lijn}
\put(9,0){\vir}
\ep
\]
\end{list}
\end{table}
\begin{theorem}
A maximal rational double point configuration on a rational
graph with almost reduced fundamental cycle occurs
in Tables \ref{tableA}, \ref{tableB} or \ref{tableC}.
\end{theorem}

\begin{proof}
By rationality at most one vertex in a rational double
point configuration can have valency three in the resolution graph.
Furthermore, a non-$(-2)$ can only be attached to a vertex with
multiplicity one in the fundamental cycle of the rational double point.
One then computes for a graph satisfying these restrictions the fundamental
cycle. The lists show that all possibilities occur.
\end{proof}

\begin{remark}
The list of configurations attached to two curves is obtained  
from the list of Table \ref{tableA} 
by replacing a vertex
with multiplicity one by a non-$(-2)$.
\end{remark}

The numbers on the graphs in the Tables
indicate  the coefficients in the fundamental cycle. The
squares are not part of the configuration, but stand for the non-$(-2)$'s,
to which the configuration is attached. The arrow indicates the curve
which intersects the fundamental cycle strict negatively.

Our notation is a combination of that in \cite{st-p} and 
Gustavsen's naming scheme \cite{gu}, which is based
on that of De Jong \cite{theo}, who gave  the list of 
Table \ref{tableA}, 
of configurations attached to  only one curve. 
Our ${}^{I}\!D_k^2$ is called $D^{\rm I}_k$ there.
Our upper indices give the multiplicity at the vertices, which are
connected to non-($-2$)'s. For the $D$-cases we could do
without the upper left $I$ or $I\!I$, except that $D^2_5$ can have two meanings.

By blowing down all RDP-configurations on the minimal resolution
$M\to X$ one obtains the \textit{canonical model}, or  \textit{RDP-resolution},
$\hat X\to X$.
The only singularities
of $\hat X$ are rational double points. 
The reduced exceptional set has two types of singularities,
normal crossing of two curves, and three curves intersecting
transversally in one point. The last case occurs
for an $A_n^{2,k,2}$ -configuration.
Again one can form a dual
graph $\hat \Gamma$, 
which in this case is a hypertree with edges for the normal
crossing points and T-joints for three curves meeting in one point.
The canonical model does not determine the 
multiplicities of the fundamental cycle on the minimal
resolution. Therefore we add this multiplicity as 
second weight (we do not write the weight  if it is 
equal to 1).

We want to draw ordinary graphs. Observe that 
given a hypertree $\hat\Gamma$  for a canonical model,
there exists a smallest ordinary
tree (i.e.,  having minimal number of
vertices) giving rise to this hypertree: one replaces each 
$T$-joint by  an $A_1^{2,2,2}$-configuration, 
i.e., by a single  $(-2)$-vertex.

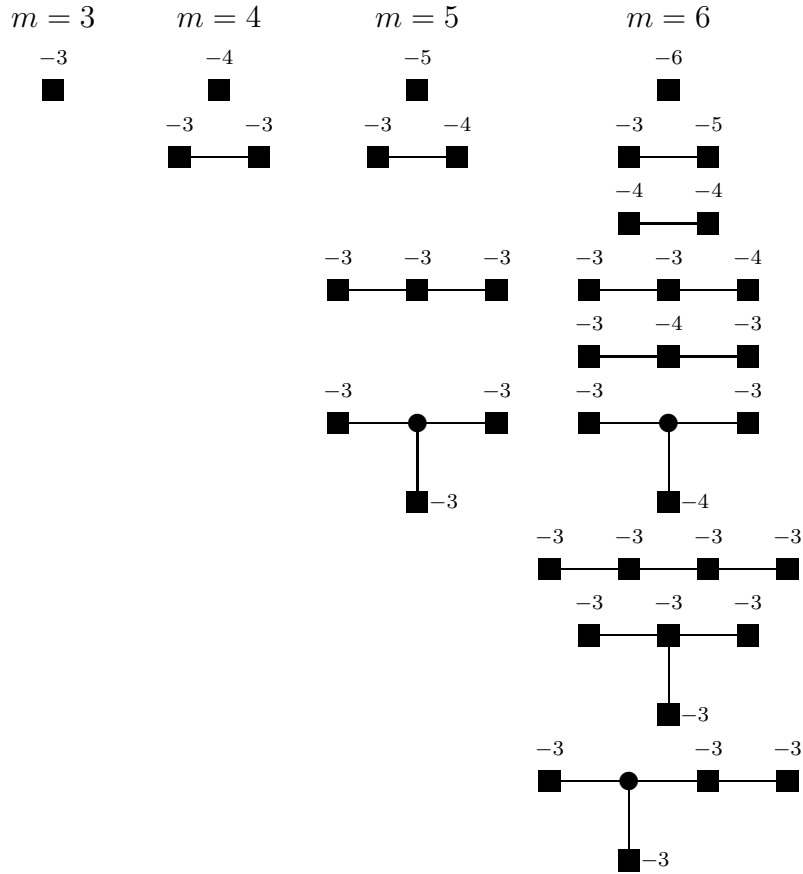
\begin{table}[h]\caption{Minimal representatives up to degree 6}\label{uptillsix}
$
\begin{array}{cccc}
\quad m=3\quad & m=4 & m=5 &m=6\\[\medskipamount]
\bpn \put(0,0){\vi{-3}}\ep &
\bpn \put(0,0){\vi{-4}}\ep &
\bpn \put(0,0){\vi{-5}}\ep &
\bp(0,.5) \put(0,0){\vi{-6}}\ep \\
&\bp(1,0) \put(0,0){\vi{-3}}\put(0,0){\lijn}\put(1,0){\vi{-3}}\ep &
\bp(1,0) \put(0,0){\vi{-3}}\put(0,0){\lijn}\put(1,0){\vi{-4}}\ep &
\bp(1,.7) \put(0,0){\vi{-3}}\put(0,0){\lijn}\put(1,0){\vi{-5}}\ep \\
&&&
\bp(1,.7) \put(0,0){\vi{-4}}\put(0,0){\lijn}\put(1,0){\vi{-4}}\ep\\
&&
\bp(2,0) \put(0,0){\vi{-3}}\put(0,0){\lijn}\put(1,0){\vi{-3}}
\put(1,0){\lijn}\put(2,0){\vi{-3}}\ep&
\bp(2,.7) \put(0,0){\vi{-3}}\put(0,0){\lijn}\put(1,0){\vi{-3}}
\put(1,0){\lijn}\put(2,0){\vi{-4}}\ep\\
&&&
\bp(2,.7) \put(0,0){\vi{-3}}\put(0,0){\lijn}\put(1,0){\vi{-4}}
\put(1,0){\lijn}\put(2,0){\vi{-3}}\ep\\
&&
\bp(2,0)(0,-1) \put(0,0){\vi{-3}}\put(0,0){\lijn}\put(1,0){\cir}
\put(1,0){\lijn}\put(2,0){\vi{-3}}
\put(1,0){\line(0,-1){1}} \put(1,-1){\vit{-3}}\ep
&
\bp(2,1.7)(0,-1) \put(0,0){\vi{-3}}\put(0,0){\lijn}\put(1,0){\cir}
\put(1,0){\lijn}\put(2,0){\vi{-3}}
\put(1,0){\line(0,-1){1}} \put(1,-1){\vit{-4}}\ep
\\
&&&
\bp(3,.7) \put(0,0){\vi{-3}}\put(0,0){\lijn}\put(1,0){\vi{-3}}
\put(1,0){\lijn}\put(2,0){\vi{-3}}
\put(2,0){\lijn}\put(3,0){\vi{-3}}\ep
\\
&&&
\bp(2,1.7)(0,-1) \put(0,0){\vi{-3}}\put(0,0){\lijn}\put(1,0){\vi{-3}}
\put(1,0){\lijn}\put(2,0){\vi{-3}}
\put(1,0){\line(0,-1){1}} \put(1,-1){\vit{-3}}\ep
\\
\hskip30pt &\hskip 50pt &\hskip 80pt &
\bp(3,1.7)(0,-1) \put(0,0){\vi{-3}}\put(0,0){\lijn}\put(1,0){\cir}
\put(1,0){\lijn}\put(2,0){\vi{-3}}
\put(1,0){\line(0,-1){1}} \put(1,-1){\vit{-3}}
\put(2,0){\lijn}\put(3,0){\vi{-3}}\ep
\end{array}
\qquad\qquad$
\end{table}

In  Table \ref{uptillsix} we list the graphs of the minimal
representatives up to degree $m=6$.
Such a graph has to have an almost reduced fundamental cycle.
The necessary and sufficient condition is that 
for a non-$(-2)$ vertex $E_i$
the sum of its valency $v(i)$ and the number of $(-2)$'s
attached to it, is  at most $b_i$.

\begin{classification}[of graphs with almost reduced fundamental cycle]
First classify all hypergraphs  of  RDP-resolutions
with all multiplicities equal to 1, and 
canonical degree $\sum (b_i-2)=m-2$. Each hypertree with
$b_i$ at least the valency of $E_i$ occurs.
Let then $\hat\Gamma$ be such a hypergraph.  
Replace a T-joint by an $A_n^{2,k,2}$
configuration,  replace any number of edges by 
configurations from Table \ref{tableB}  and attach configurations
from Table \ref{tableA} to vertices, in such a way that the
total multiplicity in the fundamental cycle
of the neighbours of any vertex $v_i$ does not exceed $b_i$.
The resulting graph is a rational graph with almost reduced
fundamental cycle, and all graphs can be obtained this way.
\end{classification}

\section{Complexity}
Lê and Tosun \cite{lt} used the number of rupture points (i.e.,
vertices with valency at least three, stars 
in the terminology of  \cite{la2}) 
as a measure of the complexity
of a rational graph. They showed that it is bounded 
in terms of the degree $m=-Z^2$ of the graph (that is, the multiplicity
of a corresponding rational singularity), more precisely
by $m-2$, if the
degree $m$ is at least 3. We give here a simplified proof for
a sharpened version. It shows that the most complex graphs are 
already obtained from singularities with reduced fundamental cycle.

\begin{defn}
The \textit{complexity} of a rational graph is the weighted number of
rupture points, where each rupture point is counted with its 
valency minus two as multiplicity.
\end{defn}

\begin{theorem}\label{comtheo}
The complexity of a rational graph of degree $m$ at least 3
is at most its canonical degree $m-2$.
\end{theorem}

The proof uses the following observation  \cite[Thm. 8]{lt}. 

\begin{lemma}
The graph, obtained from a rational graph, by making some vertex weights
more negative, is again rational and the
fundamental cycle of the new graph is reduced at the changed
vertices.
\end{lemma} 

\begin{proof}
We can obtain the new graph as subgraph of the
graph of the resolution of the original singularity, blown up in
smooth points of the relevant exceptional curves. Its
fundamental cycle can be computed by first computing the fundamental
cycle of the subgraph. By Laufer's rationality criterion the
remaining curves intersect this cycle with multiplicity one.
\end{proof}

\begin{proof}[Proof of Theorem \ref{comtheo}]

Step 1: reduction to the case of almost reduced fundamental cycle.
Consider the cycle $Y$, which has multiplicity 1 at the non-$(-2)$'s
and multiplicities on the RDP-configurations as in  Tables
\ref{tableA}, \ref{tableB} and \ref{tableC}. 
A vertex $E_i$ with $E_i\cdot Y>0$ is a non-$(-2)$
and has coefficient $z_i>1$ in the fundamental cycle. For those
$E_i$ we increase $b_i$ by one. By the previous lemma we 
get the same underlying graph
with the same complexity, but with almost reduced fundamental cycle,
namely $Y$. The contribution of $E_i$ to the canonical degree
$Z\cdot K$ changes from $z_i(b_i-2)$ to $b_i-1$ and
$(b_i-1)-z_i(b_i-2)=1-(z_i-1)(b_i-2) \leq0$ with equality if and
only if $z_i=2$ and $b_i=3$. So the degree does not increase.

Step 2: reduction to the case of reduced fundamental cycle. Consider a
RDP-configuration, where $Z$ is not reduced. Make the self-intersection
of the unique rupture point in the configuration into $-3$. This
increases the canonical degree by 1. For all non-$(-2)$'s $E_j$ to
which the configuration is connected we increase the self-intersection
by 1 (decrease $b_j$ by 1). This decreases the canonical degree by at
least 1 (here we use that $m>2$). If $b_j$ was equal to 3, then $E_j$
might be connected to at most one other RDP-configuration, but
without rupture point. The result is a longer chain of $(-2)$'s.
Proceeding in this way we obtain without increasing the degree
the same underlying graph, but with reduced fundamental cycle.

Step 3. For a graph with reduced fundamental cycle the valency of a
vertex is at most $b_i$. So the complexity is bounded by
$\sum (b_i-2)=Z\cdot K$.
\end{proof}

\section{The format of a rational singularity}
%
%
If a singularity is not a hypersurface, its equations 
can be written in many ways, some of which have a special meaning.
The standard example is the 
cone over the rational normal curve of degree four,
whose equations are the minors of
\[
\begin{pmatrix}
z_0 & z_1 & z_2 & z_3 \\
z_1 & z_2 & z_3 & z_4
\end{pmatrix}\;,
\]
but also the $2\times2$ minors of the symmetric matrix
\[
\begin{pmatrix}
z_0 & z_1 & z_2  \\
z_1 & z_2 & z_3  \\
z_2 & z_3 & z_4
\end{pmatrix}\;.
\]
In fact, perturbing these matrices gives two different ways 
of deforming the singularity, leading to the two  components
of the versal deformation.  We say that we can write the
total spaces in a determinantal format.
In a naive
interpretation a format is a way of writing or coding
(efficiently) the equations of a singularity. Another point of
view is that we have a high-dimensional variety (like the
generic determinantal), from which the singularity is derived
by specialising  the equations.
This will lead us to the definition of a format, given by Ancus 
Röhr \cite{ro}.

\begin{defn}[\cite{buc}]
Let $Y\subset\C^N$ be a singularity.
A germ $X\subset\C^M$ is a {\sl singularity of type\/} $Y$, if there
exists a map $\phi\colon \C^M\to \C^N$, such that $\phi^*(Y)=X$,
which induces a complete intersection morphism $\phi\colon X\to
Y$.
\end{defn}

For a singularity $X$ of minimal multiplicity (in particular, for a rational
surface singularity) of multiplicity at least 3 the
existence of a complete intersection morphism $X\to Y$ already
implies that $X$ is of type $Y$ \cite[2.4.2]{ro}. The singularity $Y$
has the same minimal multiplicity. Indeed, $X$ is cut out by equations
with independent linear part, for otherwise the multiplicity increases.

{\sl Deformations of type\/} $Y$ of $X$ are obtained by unfolding
the map $\phi$: for every map $ \Phi \colon \C^M\times (S,0)\to \C^N$,
extending $\phi$, the map $\pi \colon  \Phi ^* Y \to (S,0)$ is
flat \cite[4.3.4]{buc}. 
In general such deformations will not fill out a component
of the deformation space, but one can turn the problem around
and start from the total space of the deformation over a smooth
component. This total space is then rigid \cite[p.~101]{st-d}.

A rational singularity has always a smoothing component
with smooth base space. This is the \textit{Artin component}, 
over which simultaneous
resolution exists after base change. This simultaneous resolution
is a versal deformation of the resolution $M$ of $X$.
A base change is not needed, if one considers instead 
deformations of the canonical model $\hat X\to X$. 

We therefore concentrate  on the Artin component.
As it is smooth, the singularity $X$  itself is cut out by
a regular sequence from the total space $Y$ of the deformation
over the Artin component. 
Therefore the singularity is of  type $Y$.
By a result of Ephraim \cite{eph} one can write every reduced singularity $Y$
in a unique way (up to isomorphism) 
as product of a singularity $F$ and a smooth germ
of maximal dimension. 

\begin{defn}[\cite{ro}]
The \textit{format} $F(X)$ of a rational surface singularity $X$
is the unique germ $F$ in a  decomposition $Y=F\times \C^k$, with $k$ maximal,
of the total space $Y$ over the Artin component of $X$.
\end{defn} 

Let $\hat\pi\colon (\hat X,\hat Z)\to (X,p)$ be the RDP-resolution of 
a rational singularity $X$ of multiplicity $m$; it can be obtained
by blowing up a canonical ideal. It gives an embedding of
$\hat X \hookrightarrow \P^{m-2}_X$ over $X$ and with it
an embedding of the exceptional set $\hat Z=\hat \pi^{-1}(p)$
in $\P^{m-2}$, as arithmetically Cohen-Macaulay scheme of
genus 0 and degree $m-2$ \cite[2.6.3]{ro}.
R\"ohr calls the cone over $\hat Z$ the \textit{canonical cone} of $X$.
One can also obtain $\hat Z$ by blowing up a canonical ideal
of $F$. This implies that the canonical cone of a 
rational surface singularity is determined up to isomorphy by its
format. We conjectured that the converse also holds.
This would imply that the singularities in Remark \ref{formint}
have the same format.

R\"ohr proves that quasi-determinantal singularities can be
recognised from the resolution graph \cite[Satz 4.2.1]{ro}.
The condition is that the graph contains the graph of a 
cyclic quotient singularity of the same multiplicity.
Equivalently one can say that the graph $\hat \Gamma$ of the
canonical model is a chain, with everywhere
multiplicity 1. The  proof is based on a criterion
for RDP-configurations to be deformed on the resolution
without changing the format \cite[Satz 3.3.1]{ro}. This
criterion also applies to rational singularities with almost reduced
fundamental cycle: all RDP-configurations can be deformed away,
except $A_1^{2,2,2}$. 
The graph of the resulting singularity is the minimal tree for the given
hypertree 
$\hat \Gamma$. Note that the canonical cone can have moduli, so also
the formats. The graph can therefore at most determine an equisingularity
class of formats.

\section{Computation of the fundamental cycle}\label{compfc}
In this section we describe, following R\"ohr \cite[1.3]{ro}, 
a special way to compute the fundamental cycle,
for a given rational graph.  We single out a vertex $E_0$,
which we cal \textit{central vertex}.
The computation is done in steps,  where each time the
multiplicity at $E_0$ is increased by one.

We decompose the complement of  a vertex $E_0$ in a rational graph $\Gamma$  
in irreducible components:
$\Gamma\setminus \{E_0\}= \cup_{i=1}^k\Gamma_i$. We suppose
that $k>1$; the case $k=1$ can be reduced to it by blowing up a
point of the curve $E_0$.

We construct the fundamental cycle inductively. To start with,
let $E_0+Y_i^{(1)}$ be the fundamental cycle on $\{E_0\}\cup\Gamma_i$;
as $k>1$, the support of  $Y_i^{(1)}$ is $\Gamma_i$: one can
compute $Z$ starting from $E_0+Y_i^{(1)}$, so the coefficient
at $E_0$ is one. Define
$Z^{(1)}=E_0+\sum Y_i^{(1)}$. Then $Z^{(1)}\cdot E_j\leq0$ for
all $j\neq0$.

Let $Z^{(s)}$ be constructed with $Z^{(s)}\cdot E_j\leq0$ for
all $j\neq0$, with coefficient $s$ at $E_0$ and satisfying
$Z^{(s)}\leq Z$. If $Z^{(s)}\cdot E_0\leq 0$, then $Z^{(s)}$ is
the fundamental cycle $Z$. Otherwise, consider the set of vertices  
$E_{i,j}\in\Gamma_i$ 
such that $Z^{(s)}\cdot E_{i,j}=0$ and let $\Gamma_i^{(s+1)}$ be the
connected component of this set adjacent to $E_0$. Let $E_0+Y_i^{(s+1)}$
be the fundamental cycle on $\{E_0\}\cup\Gamma_i^{(s+1)}$. As
$Y_i^{(s+1)}\leq Y_i^{(1)}$, the support of $Y_i^{(s+1)}$ does
not contain $E_0$. Now define
\[
Z^{(s+1)}=Z^{(s)}+ E_0+\sum Y_i^{(s)}\;.
\]
Then $Z^{(s+1)}\cdot E_j\leq0$ for all $j\neq0$, the coefficient at
$E_0$ is $s+1$ and $Z^{(s+1)}\leq Z$; indeed $Z^{(s+1)}$ can be
constructed from $Z^{(s)}$ by first adding $E_0$ and then
continuing in the manner of a computation sequence without ever
adding $E_0$ again.

This construction ends with the fundamental cycle.

If $k=1$, we  blow up a point of the curve $E_0$,  introducing a
$\Gamma_2$. But this can be avoided, as in fact  
the same description as above holds, with the only difference that
for $k=1$ the cycle $E_0+Y_1^{(s)}$ is not the fundamental cycle 
on $\{E_0\}\cup\Gamma_1 ^{(s)}$ (in particular,
 $E_0+Y_1^{(1)}$ is not the fundamental cycle on $\Gamma$),
but $Y_1^{(s)}$ is the cycle constructed from $Z^{(s-1)}+E_0$ in the manner of a
computation sequence without ever adding $E_0$.

Let $m_i^{(s)}\leq m_i^{(1)}$ be the coefficient of $Y_i^{(s)}$ at the
vertex in $\Gamma_i$ adjacent to $E_0$. As $E_0\cdot Z^{(s)}=1$ for $s<l$,
 where $Z^{(l)}=Z$ is the last step of the
computation, 
we have $\sum _i m_i^{(1)}=b_0+1$, $\sum _i m_i^{(s)}=b_0$ for $1<s<l$ and
$\sum _i m_i^{(l)}<b_0$. 

\roep Example. Consider an $E_6$-configuration, connected
to a non-$(-2)$ vertex $E_0$. We compute the
$Z^{(s)}$. We only write the multiplicities of $E_0$ (in boldface) and 
of the
irreducible components of the configuration.
\[
\begin{matrix}\bf 1 &2&3&4&3&2\\
           &  & & 2\end{matrix}
           \qquad
\begin{matrix}\bf2 &4&5&6&4&2\cr
         &  & & 3\end{matrix}
           \qquad
\begin{matrix}\bf3 &4&5&6&4&2\cr
         &  & & 3\end{matrix}
           \]
The sequence $(m_1^{(s)})$ is $(2,2,0)$ and therefore an
$E_6$-configuration can only be connected to a curve  with
multiplicity at most 3. We observe that the same sequence can be
obtained from $2A_2^1$, two chains of length two.

\section{One non-reduced curve}
\label{onrnmt}
The goal of this section is to give the elements for the classification of rational
graphs, where each RDP-configuration is attached to at most one 
non-reduced non-$(-2)$-vertex. 
We first  classify the possible multiplicities at RDP-configurations.
These depend only on the multiplicity of the non-$(-2)$,
and the computation can again be done for each
configuration separately.
The candidates for graphs of RDP-resolutions can be found from the
graphs with almost reduced fundamental cycle, but not every
candidate arises from a rational graph.

Let $E_0$ be a non-reduced non-$(-2)$, with multiplicity
$z_0$ in the fundamental cycle. According to the  previous
section, we can compute the fundamental cycle in
$z_0$ steps, each time increasing the multiplicity of $E_0$ by
one. We add cycles with support on the subgraphs $\Gamma_i$
and each subgraph gives a multiplicity sequence
$(m_i^{(1)},\dots,m_i^{(z_0)})$.
These multiplicities satisfy 
\[\sum _i m_i^{(1)}=b_0+1, 
\qquad\sum _i m_i^{(s)}=b_0\text{ for }1<s<z_0, \qquad
\sum _i m_i^{(z_0)}<b_0.
\] 
After the first step we add only cycles with support in 
RDP-configurations intersecting $E_0$, as all other
non-$(-2)$-curves, intersecting such configurations,
have multiplicity one. 
Each $\Gamma_i$ contains at most one RDP-configurations adjacent $E_0$.
We include the case that there is no such configuration by calling it
$A^{1,1}_0$.

For each RDP-configuration from Tables \ref{tableA},  \ref{tableB}
and \ref{tableC} we compute the \textit{multiplicity
sequence} $(m^{(1)},\dots,m^{(j)})$. The multiplicities satisfy
$m^{(1)}-1\leq m^{(s)} \leq m^{(1)}$ for all $s<j$.
We abbreviate a sequence $k,\dots,k$ of $l$ equal 
multiplicities as $k^l$. An exponent $l=0$ means that this factor is absent.
If the sequence is infinite, and repeating itself,
we underline the repeating section. So in Table \ref{tableBB}
the entry $(1^{n+1},\underline{0,1^n})$ for
$LA_n^{1,1}$ should be read as
$(1^{n+1},0,1^n,0,1^n,0,\dots)$. The case $n=0$, of two non-$(-2)$'s
intersecting each other, is included. The sequence is then $(1,0,0,\dots)$.

For configurations between several vertices
only one of the non-$(-2)$'s has higher multiplicity, and
we suppose  that the other ones have 
sufficiently negative self-intersection for the graph being rational.

We have to distinguish which of the two or three attached vertices
is the non-reduced one. 
We always draw the graphs as in Tables \ref{tableB} and \ref{tableC}.  
In a graph of type
${}^{\it I}\!A_n^{2,k}$,  ${}^{\it II}\!A_n^{k,2}$ or
$A_n^{2,k,2}$ the arrowhead (which indicates the curve intersecting
the fundamental cycle of the extended configuration negatively) is
on the right hand side of the graph. So it makes sense to
distinguish between the left, middle or right attached vertex.
We denote this by writing an $L$, $M$ or $R$ before the name.
For type $D$ we use $L$ and $R$.

It is possible to obtain a multiplicity sequence from different
configurations or combinations of configurations. We then speak
about equivalent configurations. For each configuration we also
determine the simplest equivalent combination of configurations.

\begin{prop} \label{proponmt}
The multiplicity sequences and the equivalent
configurations for  the configurations of Table \ref{tableA}
are as given in  Table \ref{tableAA}.
The different cases arising from the configurations of Table \ref{tableB}
are  in   Table \ref{tableBB}; it gives also the multiplicity
at the component attached to the other, reduced non-$(-2)$.
If the sequence is infinite, the multiplicity after step $s$
of the computation is given.
 Table \ref{tableCC}
gives the results for $A_n^{2,k,2}$.
\end{prop}
\begin{table}[h]\caption{}
\label{tableAA}
\renewcommand{\arraystretch}{1.2}
$
\begin{array}{llc}
\mbox{name} &\mbox{mult sequence}   &  \mbox{equivalent to}\\
\hline
\\[-1.2\bigskipamount]
A^1_n &  (\underline{1^n,0}) & \\
A^k_{(l+1)k+r-1}, \; r <k-1  &(k^l,r) & (k-r)A_l^1+rA^1_{l+1} \\
A^2_{2l+2} & (2^l,1,1,2^l,0) & \\
A^k_{(l+2)k-2}, \; k>2  &(k^l,k-1,1) &A_l^1+(k-1)A^1_{l+1}\\
^{\it I}\!D_k^2     & (2,0)&   2A_1^1 \\
^{\it II}\!D_{2k}^k,  \; k>2  &  (k,0) &kA_1^1 \\
^{\it II}\!D_{5}^2 &(2,1,2,0)  & A_1^1+A_3^1 \\
^{\it II}\!D_{2k+1}^k, \; k>2  &(k,1) & (k-1)A_1^1+A_2^1 \\
E_6^2      &  (2,2,0) & 2A_2^1 \\
E_7^3      & (3,0) &  3A_1^1 \\ 
\hline
\end{array}
$
\end{table}

\begin{table} \caption{} \label{tableBB}
\renewcommand{\arraystretch}{1.2}
$
\begin{array}{llll}
\mbox{name}& \mbox{mult sequence}& \mbox{other mult}& 
\mbox{equivalent configuration} \\ 
\hline
\\[-1.2\bigskipamount]
 LA_n^{1,1}
  & (1^{n+1},\underline{0,1^n})&
  \lceil \frac{n+s}{n+1} \rceil 
\\
L^{\it I}\!A_n^{2,k} &(2,1^{n-k},0)& 
  n\!-\!k\!+\!2 & LA_0^{1,1}+A_{n-k+1}^1
\\
M^{\it I}\!A_{(l+1)(k-1)+r}^{2,k} 
& (k,(k-1)^{l-1},r)  & \lceil \frac{(l+1)(k-1)+r}{k-1} \rceil
&  LA_0^{1,1}\!+\!( k\!-\!1\!-\!r)A_{l}^1\!+\!rA_{l+1}^1\\
\quad k>2
\\
M^{\it II}\!A_{(l+1)k+r-2}^{k,2} 
 & (k^l,r) &2 & 
LA_{l-1}^{1,1} \!+\!rA_{l+1}^1\!+\!( k\!-\!1\!-\!r)A_{l}^1\\
\quad 0\leq r <k\!-\!1\\
M^{\it II}\!A_{(l+1)k+k-3}^{k,2} 
     & (k^l,k-1,1)  &3 & 
LA_{l-1}^{1,1}+(k-1)A_{l+1}^1\\
\quad k>2,\; l>1\\
M^{\it II}\!A_{3k-3}^{k,2}  
      & (k,k-1,1)  &3 & 
LA_{1}^{1,1}+A_1^1+(k-2)A_{2}^1\\
\quad k>2\\
M^{\it II}\!A_{2l+1}^{2,2}
  &(2^l,\underline{1^2,2^{l-1}})&
\lceil \frac{l+1+s}{l+1} \rceil 
&
LA_{l}^{1,1}+A_l^1
\\
R^{\it II}\!A_n^{k,2}
  &(2,1^{k-2},0) &k&
LA_0^{1,1}+A_{k-1}^1\\
\quad k>2\\
R^{\it II}\!A_n^{2,2} &(2,0) &2&
LA_0^{1,1}+A_{1}^1
\\
LD_{2k+1}^{k+1,2}& (k+1,0) & 2 & LA_0^{1,1}+kA_1^1
\\
RD_{2k+1}^{k+1,2}& (2,1,1,\dots)  & 
        \lceil \frac{2k+s}2 \rceil 
&
  LA_1^{1,1}+A_1^1
\\
LD_{2k}^{k,2}
   & (k,1) & 3 & (k-2)A_1^1+A_2^1+LA_0^{1,1}\\
\quad k>2\\
RD_{2k}^{k,2} &(2,1,1,\dots) & 
        \lfloor \frac{2k+s}2 \rfloor 
&
  LA_1^{1,1}+A_{1}^1
\\ 
\hline
\end{array}
$
\end{table}

\begin{table} \caption{} \label{tableCC}
\renewcommand{\arraystretch}{1.2}
$
\begin{array}{lllll}
\mbox{name}& \mbox{mult seq}& \mbox{mult at L}& 
\mbox{mult at M} &\mbox{mult at R}\\ 
\hline
\\[-1.2\bigskipamount]
LA_n^{2,k,2}& 
 (2,1^{n-k+1},0) && n-k+3 &2
\\
MA_n^{2,k,2} &(k,(k-1)^{l-1},r) &\lceil \frac{n+1}{k-1} \rceil && 2
\\
RA_n^{2,k,2} & (2,1^{k-2},0) & 2 & k\\
\hline
\\[-\bigskipamount]
\multicolumn{5}{c}{$Here
$n=(l+1)(k-1)-1+r$ with $0\leq r\leq k-2$ and $k>2$.$}
\end{array}
$
\end{table}

\begin{proof}
We do here only the case $A_n^k$, for  $k>1$, as the other cases involve similar or easier
computations. We write $n=lk+r+(k-1)$ with $l\ge1$ and $0\le r\le k-1$. This is possible
as the number $n$ satisfies $n\ge 2k-1$. 
There is a chain of
$lk+r-(k-1)=(l-1)k+r+1$ $(-2)$-vertices with multiplicity $k$ in $Z^{(1)}$,
and the end of this chain not intersecting $E_0$ 
intersects $Z^{(1)}$ negatively
(when $l=1$ and $r=0$ there is only one vertex with
multiplicity $k$; in this case the multiplicity sequence  is $(k,0)$
and the format is $kA_1^1$, in accordance with the general
formula). The set $\Gamma^{(2)}$ consists of
$(l-1)k+r+(k-1)$  vertices. If $l>1$ this number is at least $2k-1$ and
the multiplicities in $Z^{(2)}$ are 
\[
2,4,\dots,2k,2k,\dots,2k,
2k-1,2k-2,\dots,2,1\;.
\] 
There are $(l-2)k+r+1$ vertices with multiplicity
$2k$ in $Z^{(2)}$. We continue in this way until there are $r+1$
vertices with multiplicity $lk$ in $Z^{(l)}$; all multiplicities are
then 
\[
l,2l,\dots,lk,lk,\dots,lk, lk-1,lk-2,\dots,2,1\;.
\] 
The set
$\Gamma^{(l+1)}$ consists of $r+(k-1)$  vertices (except when $r=0$;
then $\Gamma^{(l+1)}$ is empty).
We therefore add the multiplicities 
\[
1,2,\dots,
r-1,r,\dots,r,r-1,\dots,2,1,0,\dots,0\;.
\]
If $r<k-1$ the sequence
stops here, the multiplicity sequence  is $(k^l,r)$ and the
equivalent configuration is $(k-r)A_l^1+rA_{l+1}^1$. 
If $r=k-1$ the multiplicities in
$Z^{(l+1)}$ are 
\[
l+1, 2(l+1),  \dots, (k-1)(l+1),k(l+1)-1,
k(l+1)-2,\dots, 2,1\;.
\]
We add the multiplicities
$0,\dots,0,1,\dots,1$. If $k\ge3$ the sequence stops here, the
multiplicity sequence is $(k^l,k-1,1)$ and the configuration is
equivalent to $A_l^1+(k-1)A_{l+1}^1$. If $k=2$, the sequence continues; as
$\Gamma^{(l+3)}$ consists of $n-1$  nodes,  the multiplicity sequence 
is $(2^l,1^2,2^l,0)$. 
There is no easier equivalent configuration for this $A^2_{2l+2}$.
\end{proof}

\begin{remark}
The condition $k>2$ in the tables is included to avoid duplications.
For example, as $ MA_n^{2,2,2}=LA_n^{2,2,2}$, we can assume that $ k>2$
for  $MA_n^{2,k,2}$.
\end{remark}

\begin{remark}
Note that the tables give the maximal multiplicity sequence
for each configuration. If the computation stops earlier (due to
other configurations), one gets a simpler equivalent singularity.
\end{remark}

\begin{cor}\label{cor_alleen_A}
Every RDP-configuration, attached to only one vertex,
is equivalent to a combination of configurations of type
$A_n^1$ and $A_{2l}^2$.
\end{cor}

\begin{cor}
An RDP-configuration, attached to two or three vertices, of which only
one has  multiplicity greater than one in the fundamental cycle,
is equivalent to a combination of configurations  of type
$A_n^1$, $A_{2l}^2$, $LA_n^{1,1}$ and $LA_{n}^{2,2,2}$.
\end{cor}

\begin{proof}
Table \ref{tableBB} gives the  result for configurations between two vertices.

From Table \ref{tableCC} we see that the multiplicities of $LA_n^{2,k,2}$
depend only on $n-k$, so   $LA_n^{2,k,2}$ is equivalent to
$LA_{n-k+2}^{2,2,2}$. The multiplicities of $RA_n^{2,k,2}$
depend only on $k$, so we can take the smallest $n$, which is $2k-3$.
In that case the left and right chain of $(-2)$'s are equally long, so by
interchanging $L$ and $R$ we obtain $LA_{2k-3}^{2,k,2}$, which
is equivalent to $LA_{k-1}^{2,2,2}$.

For $MA_n^{2,k,2}$ we distinguish between the cases $r=0$ and
$0<r\leq k-2$. In the first case 
$\lceil \frac{n+1}{k-1} \rceil = l+1$, while 
$\lceil \frac{n+1}{k-1} \rceil = l+2$ in the second case.
For $r=0$ an equivalent configuration, attached to the
vertex $v_M$, is $MA_l^{2,2,2}+(k-2)A_l^1$, and for
$r>0$ it is $MA_{l+1}^{2,2,2}+(k-1-r)A_l^1+(r-1)A_{l+1}^1$.
Finally we note that interchanging $M$ and $L$ makes   
$MA_n^{2,2,2}$ into  $LA_n^{2,2,2}$.
\end{proof}

From an arbitrary rational graph we obtain
a graph with almost reduced fundamental cycle and the same
underlying graph by making some vertex weights $-b_i< -2$ more negative. 
This process can also be inverted. 
The possible candidates for graphs (or hypergraphs) of
RDP-resolutions with non-reduced fundamental cycle can be obtained
from reduced (hyper)-graphs by replacing a $-(b_i+2)$-vertex by a
$-(b_i/z_i+2)$-vertex with multiplicity $z_i$, but  not
all graphs can be realised.

\begin{prop}\label{propsix}
On a rational graph with only one non-$(-2)$ vertex $E_0$ the
multiplicity of $E_0$ in the fundamental cycle can at most be $6$.
\end{prop}

\begin{proof}
By Corollary \ref{cor_alleen_A} it suffices to consider only 
RDP-configurations of type $A_n^1$ and
$A_{2l}^2$. If $z_0>2$, there is exactly one configuration
$\Gamma_i$ with $m_i^{(2)}=m_i^{(1)}-1$, so it is either $A_1^1$ or
$A_4^2$. In the last case $z_0\le 5$, as $A_4^2$ gives the sequence
$(2,1,1,2,0)$. Suppose now that there is exactly one $A_1^1$. If
$z_0>3$, there is exactly one configuration $\Gamma_i$ with
$m_i^{(3)}=m_i^{(2)}-1=m_i^{(1)}-1$, which is either $A_2^1$ or $A_6^2$.
In the last case $z_0=4$, as we combine the sequences
$(2,2,1,1,2,2,0)$ and $(1,0,1,0,\dots)$. The sequence of
$A_1^1+A_2^1$ is $(1+1,0+1,1+0,0+1,1+1,0+0)=(2,1,1,1,2,0)$, which
shows that $z_0\leq 6$.
\end{proof}

\begin{remark}  We realise $z_0=6$ for a $(-3)$ with
$A_1^1+A_2^1+A_4^1+A_5^1$.
\end{remark}

\begin{remark}\label{formint}  With $A_4^2$ and $E_0$ a
$(-3)$ we can realise
$z_0=5$ with the configuration $A_4^2+A_3^1+A_4^1$. Another way
to get $5E_0$ is with $A_1^1+A_2^1+2A_4^1$. 
It would be interesting to study the formats of the corresponding 
singularities.
We remark that neither is a deformation of the other.
\end{remark}

\begin{classification}[of graphs, where
each RDP-configuration is attached to at most one non-reduced
non-$(-2)$]
Start by making a list of all possible
hypergraphs $\hat\Gamma$ of canonical cones, without edges
(or hyperedges) between  non-reduced
vertices.
Given $\hat\Gamma$, realise this graph (if possible)
in all ways, using only configurations $A_n^1$ and $A_{2l}^2$,
$A_n^{1,1}$ (including $n=0$)  and
$LA_{n}^{2,2,2}$.
Replace (combinations of) RDP-configurations with 
equivalent ones, as given by the Tables \ref{tableAA}, \ref{tableBB} and \ref{tableCC}.
\end{classification}

\begin{remark}
The computations so far also can help to compute the fundamental cycle
for complicated graphs.
As example we return to  Karras' graph, given at the end of
Section \ref{sect_een}.  The graph for the canonical
model is rather simple. Note also that only configurations of type
$A_n^1$ occur.
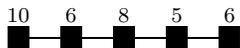
\[
\unitlength=20pt
\bp(4,.5)(9,0)
\put(9,0){\vi{10}} \put(9,0){\lijn}
\put(10,0){\vi{6}} \put(10,0){\lijn}
\put(11,0){\vi8} \put(11,0){\lijn}
\put(12,0){\vi5} \put(12,0){\lijn}
\put(13,0){\vi6} 
\ep
\]
We first simplify the graph.
The configuration 
$A_1^1+A_2^1$ at $(-3)$ on the right implies that its multiplicity
is at most $6$. Therefore the $A_5^1$ has no influence on the 
computation, and we get the same multiplicities, if we remove it
and increase the weight $(-3)$ to  $(-2)$. We have then a
$A_4^2$ attached to the $(-3)$ of multiplicity $5$.
The $(-3)$ on the left has multiplicity at most $10$ because of
$A_1^1+A_{9}^1$. Again we can remove the $ A_{9}^1$ and increase
the weight $(-3)$ to  $(-2)$.  We have then a
$A_{10}^2$ attached to the $(-3)$ of multiplicity $6$. By the same 
argument the $A_7^1$ at the vertex of multiplicity $8$ can be 
removed, so that we end up with two $(-3)$-vertices $E_1$ and $E_2$
with a $A_3^{2,2}$ in between, an $A_{10}^2$ attached to $E_1$
and $A_4^2$  attached to $E_2$.

It remains to compute the
fundamental cycle for this configuration.
This is best done with the  rupture point between $E_1$ and $E_2$
as central vertex.
We give the total multiplicities at each step. The multiplicities
of the $(-3)$'s are in bold face, while those of the central
vertex are underlined.
{\def\it{\underline}
\begin{gather*}\allowdisplaybreaks
\setcounter{MaxMatrixCols}{20}
\begin{matrix}
  &    &     &   &    &    &   &    & 1 &         &   &     1 &        &  1         \\
1 & 2 & 2  & 2 & 2 & 2 & 2 & 2 & 2 & \bf 1 & 1 & \it 1 & \bf 1 & 2 & 2 & 1
\end{matrix}
\\[\medskipamount]
\begin{matrix}
  &    &     &   &    &    &   &    & 2 &         &   &     1 &        &  2         \\
1 & 2 & 3  & 4 & 4 & 4 & 4 & 4 & 4 & \bf 2 & 2 & \it 2 & \bf 2 & 3 & 2 & 1
\end{matrix}
\\[\medskipamount]
\begin{matrix}
  &    &     &   &    &    &   &    & 3 &         &   &     2 &        &  2         \\
1 & 2 & 3  & 4 & 5 & 6 & 6 & 6 & 6 & \bf 3 & 3 & \it 3 & \bf 2 & 3 & 2 & 1
\end{matrix}
\\[\medskipamount]
\begin{matrix}
  &    &     &   &    &    &   &    & 4 &         &   &     2 &        &  2         \\
1 & 2 & 3  & 4 & 5 & 6 & 7 & 8 & 8 & \bf 4 & 4 & \it 4 & \bf 3 & 4 & 3 & 2
\end{matrix}
\\[\medskipamount]
\begin{matrix}
  &    &     &   &    &    &   &    &  5  &       &   &     3 &        &  2         \\
1 & 2 & 3  & 4 & 5 & 6 & 7 & 8  & 9 & \bf 5 & 5 & \it 5 & \bf 3 & 4 & 3 & 2
\end{matrix}
\\[\medskipamount]
\begin{matrix}
  &    &     &   &    &    &   &    &  5  &       &   &     3 &        &  3         \\
1 & 2 & 3  & 4 & 5 & 6 & 7 & 8  & 9 & \bf 5 & 6 & \it 6 & \bf 4 & 6 & 4 & 2
\end{matrix}
\\[\medskipamount]
\begin{matrix}
  &    &     &   &    &    &   &    &  5  &       &   &     4 &        &  3         \\
1 & 2 & 3  & 4 & 5 & 6 & 7 & 8  & 9 & \bf 5 & 6 & \it 7 & \bf 5 & 6 & 4 & 2
\end{matrix}
\\[\medskipamount]
\begin{matrix}
  &    &     &   &    &    &   &    &  5  &       &   &     4 &        &  3         \\
2 & 3 & 4  & 5 & 6 & 7 & 8 & 9  & 
\makebox[0pt]{10} & \bf 6 & 7 & \it 8 & \bf 5 & 6 & 4 & 2
\end{matrix}
\end{gather*}
}
\end{remark} 

\section{RDP-configurations on general graphs}
In this section  we determine  the maximal multiplicities 
that  most can occur on an
RDP-configuration. 
We continue to compute for each RDP-configuration separately.
For some configurations the multiplicities can become arbitrary
high, but what actually happens, depends on the rest of the graph.
We do  investigate the exact conditions. 

The results apply to the classification of graphs, in which
two or three non-reduced non-$(-2)$'s are connected to each other
by a single RDP-configuration, but not connected to any other
 non-reduced non-$(-2)$. In particular, we determine the conditions
that the multiplicity  of the non-$(-2)$'s does not exceed two.
This suffices  to give a complete classification of rational
graphs of degree 6. We  indicate this in the next section.

We first treat configurations attached to exactly two vertices, both of 
higher multiplicity.  
Then there are two vertices $E_a$ and 
$E_b$, of self-intersection $-a$ and $-b$, which are connected by 
a RDP-configuration $\Delta$. 
The fundamental cycle $E_a+Z_\Delta+E_b$ on 
$\{E_a\}\cup\Delta\cup\{E_b\}$ is given in Table \ref{tableB}.
Let $n_{\Delta,a}$ be the coefficient of the vertex of $\Delta$, 
adjacent to $E_a$.
Furthermore, let $\Gamma_a$ be the union of the 
connected components of the complement of the graph, which are connected
to $E_a$.    Let $E_a+Z_a$ be the
fundamental cycle on $\{E_a\}\cup \Gamma_a$,  
let $n_a$ be the sum of the
multiplicities of $Z_a$ at the vertices of $\Gamma_a$, adjacent to $E_a$.
Define the corresponding objects for $E_b$.

\begin{defn}
In the above situation  $E_a$ is a \textit{bad vertex} if 
$n_{\Delta,a}+n_a=a+1$.
\end{defn}

We borrow the term \textit{bad} from Tosun,
see \cite[Definition 3.4]{lt} and \cite[Definition 3.14]{tosun}, where
it is used without multiplicities: Tosun calls a vertex bad
if its valency is one less then its  vertex weight $b$.
Karras \cite{ka} calls it a \textit{basic center}.
If $E_i\cdot Z_\Delta=0$ for every vertex  of $\Delta$,
then exactly one of $E_a$ and $E_b$ is a bad vertex (in our sense).


\subsection{$A_n^{1,1}$}
We call the two vertices $E_L$ and $E_R$, and denote the numbers
defined above correspondingly;  the
vertex weight of $E_L$ is $-b_L$, and that of $E_R$ is $-b_R$. Then
$n_{\Delta,L}=n_{\Delta,R}=1$ and there is exactly one bad vertex, 
which we suppose 
to be $E_L$. This means that $n_L=a$, and $\Gamma_L$ is non-empty. 
We claim that the multiplicity of
$\Delta$ in the fundamental cycle  can be arbitrarily high.
We compute the fundamental cycle  with $E_L$ as central
vertex.  
We set $Y_L^{(1)}=Z_L$,  $Y_{\Delta,R}^{(1)}=Z_\Delta+E_R+Z_R$.
Then $Z^{(1)}=Y_L^{(1)}+E_L+Y_{\Delta,R}^{(1)}$, and $E_L$ is the only
vertex with $E_i\cdot Z^{(1)}=1$. In each next step $Y_L^{(s)}\leq Y_L^{(1)}$
and $Y_{\Delta,R}^{(s)}\leq Y_{\Delta,R}^{(1)}$. In particular, the multiplicity
of the fundamental cycle at $E_R$ does not exceed that at $E_L$.
We describe the case that the computation never stops.
For the sum   $n_L^{(s)}$ of multiplicities in $\Gamma_L$, adjacent
to $E_L$, and the multiplicity $n_{R,\Delta}^{(s)}$ we have then either
$n_L^{(s)}=a-1$ and $n_{\Delta,R}^{(s)}=1$, or
$n_L^{(s)}=a$ and $n_{\Delta,R}^{(s)}=0$. As remarked earlier,
we do not investigate the conditions which this  assumption imposes on
$\Gamma_L$ and $\Gamma_R$.

Let $Z_\Delta=E_1+\cdots+E_n$ with $E_1\cdot E_L=1$ and  $E_n\cdot E_R=1$.
Suppose the coefficient of $E_R$ in  $Y_{\Delta,R}^{(s)}$ is $1$, and the
coefficient of $E_R$ in $Z^{(s)}$ is $k$. 
If $E_R\cdot Z^{(s)}=-s_k<0$, then
$Y_{\Delta,R}^{(s+1)}=E_1+\cdots+E_n$,
$Y_{\Delta,R}^{(s+2)}=E_1+\cdots+E_{n-1}$, \dots,
$y_{\Delta,R}^{(s+n)}=E_1$ and $Y_{\Delta,R}^{(s+n+1)}=
\emptyset$. Then $E_R\cdot Z^{(s+n+1)}=-s_k+1$.
We continue by adding only cycles with support on $\Delta$ until $E_R$
intersects the total computed cycle trivially. In the next step
the coefficient of $E_R$ in the added cycle will again be 1.
At this stage the coefficients of the total cycle in the neighbourhood
of $\Delta$ are as follows. 

The coefficient of $E_L$ is $s=k+(n+1)\sum s_i$, the sum of the $n_L^{(j)}$
is $(a-1)(n\sum s_i+k-1)+a(1+\sum s_i)$,
the coefficient of $E_1$ is $k+n\sum s_i$, that of $E_t$ is
$k+(n+1-t)\sum s_i$, that of $E_n$ is $k+\sum s_i$, the coefficient of
$E_R$ is $k$, and the sum of the multiplicities of the vertices in $\Gamma_R$,
adjacent to $E_R$, is $k(b-1)-\sum s_i$.

We remark that the formulas also work, if $n=0$. This means that 
$\Delta=\emptyset$ and $E_L$ is adjacent to $E_R$.
Furthermore, if $\sum s_i=0$, the multiplicities at $E_L$ and $E_R$ are
independent of $n$.  

\subsection{${}^{\it I}\!A_n^{2,k}$}
In this case, and also for ${}^{\it II}\!A_n^{k,2}$ and
$A_n^{2,k,2}$, it is more convenient to compute the fundamental
cycle with the rupture point in the chain of $(-2)$'s 
as central vertex $E_0$. We therefore use 
a slightly different notation, consistent with the description of the
computation in Section \ref{compfc}.  
Let  $m_L^{(s)}$,  $m_M^{(s)}$  and $m_R^{(s)}$
be the multiplicities in step $s$ at the vertices directly to the left, below or
to the right of the central vertex. 
The non-$(-2)$ vertices are $E_L$ with weight $-b_L$, and $E_M$ with weight
$-b_M$. 

We have $m_L^{(1)}+ m_M^{(1)}+m_R^{(1)}=3$, 
$m_L^{(s)}+ m_M^{(s)}+m_R^{(s)}\leq 2$ for $s>1$, and the
computation stops at the first $s$ where this sum is less than 2.

We start by computing the sequence $(m_R^{(s)})$. We apply
Proposition \ref{proponmt}: 
as we have a $A^1_{n-k+1}$-configuration, the sequence
is $(1^{n-k+1},0,1^{n-k+1},0,\dots)$. So $m_R^{(s)}=1$ for 
$s\neq l(n-k+2)$ and $m_R^{(l(n-k+2))}=0$ for all $l$.

Next we look at $(m_M^{(s)})$. Let $Z_M$ be the fundamental cycle
on the connected component $\Gamma_M$ of $\Gamma\setminus \{E_0\}$,
containing $E_M$. 
For the first step  $Z^{(1)}$ of the computation we determine the fundamental cycle 
$Y_M^{(1)}$ on $\{E_0\}\cup  \Gamma_M$: it is $E_0+Z_M$. The
condition that $E_M$ is a bad vertex  translates into $E_M\cdot Z_M=1-k$,
so $E_M\cdot Z^{(1)}=2-k$. Therefore we put $E_M\cdot Z^{(1)}=2-k-t_1$,
where $t_1\geq0$ with equality if and only if $E_M$ is a bad vertex. In the 
next steps $Y_M^{(s)}$ is empty. We find that $E_M\cdot Z^{(k+t_1-1)}=0$,
so $E_M$ is in the support of $Y_M^{(k-t_1)}$. We set
$E_M\cdot Z^{(k-t_1)}=2-k-t_2$, with $t_2\geq t_1$. Proceeding this way
we find the sequence 
\[
(1,0^{k+t_1-2},1,0^{k+t_2-2},1,0^{k+t_3-2},\dots)\;.
\]

On the left side $E_L\cdot Z^{(1)}=-s_1$ with $s_1=0$
if and only if $E_L$ is a bad vertex. If $s_1>0$, then $E_L$ is not contained
in the support of $Y_L^{(2)}$, which is the $A_{k-2}$-configuration
between $E_L$ and
$E_0$. We continue in the manner of $A_{k-2}^1$, until 
$E_L\cdot Z^{(s_1(k-1)+1)}=0$ and $E_L$ is in the support of
$Y_L^{(s_1(k-1)+2)}$. Then 
$E_L\cdot Z^{(s_1(k-1)+2)}=-s_2$ with $s_2\geq s_1$.
The sequence is
 \[
(1,(1^{k-2},0)^{s_1},1,(1^{k-2},0)^{s_2},1,(1^{k-2},0)^{s_3},\dots)\;.
\] 

Exactly one of $E_L$ and $E_M$ is a bad vertex. 
If $k=2$, both $E_L$ and $E_M$ are connected to $E_0$, so upon relabeling
we may assume that the bad vertex is $E_L$. We first treat the
other case, that $E_M$ is the bad vertex. Then $t_1=0$, and, as just said, 
we make the assumption  that $k>2$.
To obtain a high multiplicity we need that  
$m_L^{(s)}+ m_M^{(s)}=1$ for $1<s<n-k+2$, and 
$m_L^{(n-k+2)}+ m_M^{(n-k+2)}=2$.
We achieve $m_L^{(s)}+ m_M^{(s)}=1$ for $s>1$
by taking $t_1=\dots=t_{s_1}=0$, 
$t_{s_1+1}=1$,   $t_{s_1+2}=\dots=t_{s_1+s_2}=0$,
$t_{s_1+s_2+1}=1$,   $t_{s_1+s_2+2}=\dots=t_{s_1+s_2+s_3}=0$,
\dots.
The only possibility to get $m_L^{(s)}= m_M^{(s)}=1$
is by taking $t_{s_1+\dots+s_p+1}=0$:
this gives $s=p+\sum_{i=1}^ps_i(k-1)+k-1$.
We therefore put $n-k+2=p+\sum_{i=1}^ps_i(k-1)+r$ with
$r<1+s_{p+1}(k-1)$. If $r\neq k-1$, the computation stops
with $s=n-k+2$. If $r=k-1$, we go one step further, as then
$m_L^{(n-k+2)}= m_M^{(n-k+2)}=1$ and $m_R^{(n-k+2)}=0$,
but $m_L^{(n-k+3)}= m_M^{(n-k+3)}=0$. So the computation
always stops.

Suppose now that $E_L$ is the bad vertex; here $k=2$ is allowed.
In this case the computation need not end. We have $s_i=0,1$ for
all $i$.
As $m_M^{(p(k-1)+\sum t_i +1)}=1$,
we obtain $m_L^{(s)}+ m_M^{(s)}=1$ for $s>1$
by taking $s_1=\dots=s_{t_1}=0$, 
$s_{t_1+1}=1$,   $s_{t_1+2}=\dots=s_{t_1+t_2}=0$,
$s_{t_1+t_2+1}=1$,   $s_{t_1+t_2+2}=\dots=s_{t_1+t_2+t_3}=0$,
\dots. 
We need $m_L^{(s)}+ m_M^{(s)}=2$ for $s=l(n-k+2)$.
This is possible if $p(k-1)+\sum_{i=1}^m  t_i +1=l(n-k+2)$.
In case  $l=1$ we then do not set $s_{\sum t_i +1} =1$, but continue
with $s_{\sum t_i +1} =s_{\sum t_i +2} =\dots=0$. This gives a shift in the
indices of the $s_i$, which we do not compute here.

\subsection{${}^{\it II}\!A_n^{k,2}$}
In this case the sequence $(m_L^{(s)})$
is $(1^{k-1},0,1^{k-1},0,\dots)$. So $m_R^{(s)}=1$ for 
$s\neq lk$ and $m_R^{(lk)}=0$ for all $l$.
As in the previous case the sequence  $(m_M^{(s)})$ is
\[
(1,0^{k+t_1-2},1,0^{k+t_2-2},1,0^{k+t_3-2},\dots)\;.
\]

We have $E_R\cdot Z^{(1)}=-u_1$ with $u_1=0$
if and only if $E_R$ is a bad vertex. If $u_1>0$, then $E_R$ is not contained
in the support of $Y_R^{(2)}$, which is the $A_{n-k}$ between $E_R$ and
$E_0$. 
The sequence $(m_R^{(s)})$ is
 \[
(1,(1^{n-k},0)^{u_1},1,(1^{n-k},0)^{u_2},1,(1^{n-k},0)^{u_3},\dots)\;.
\] 

We achieve  that $m_L^{(s)}+ m_M^{(s)}=1$ for $s>1$
by taking $t_1=0$ and $t_i=1$ for $i>1$. It is possible to
have $m_L^{(lk)}= m_M^{(lk)}=1$ for some $l>1$, while
$m_L^{(pk)}+ m_M^{(pk)}=1$ for $p<l$, by setting 
$t_{l}=0$. If $k>2$,  then $m_L^{(lk+1)}= m_M^{(lk+1)}=0$,
so the computation stops at that point.
Therefore the computation stops  when $m_R^{(s)}$=0, or if
$s=lk$, in the next step.
The computation never stops if $u_i=0$ for all $i$. Note that
in that case both $E_M$ and $E_R$ are bad vertices.

If $k=2$, the situation is a bit different.
The  sequence $(m_L^{(s)})$
is $(1,0,1,0,\dots)$,  $(m_M^{(s)})$ is
$(1,0^{t_1},1,0^{t_2},1,0^{t_3},\dots)$ and 
$(m_R^{(s)})$ is
 \[
(1,(1^{n-2},0)^{u_1},1,(1^{n-2},0)^{u_2},1,(1^{n-2},0)^{u_3},\dots)\;.
\]
We always take $t_1=0$, and $t_i\leq 1$. 
By taking suitable consecutive $t_i$ equal to zero
we can get $m_L^{(2s)}+ m_M^{(2s)}=2$, with this sum always equal to
one for odd indices. It is possible that the computation never stops.
If $n$ is odd, we need $u_{2l-1}=0$ for all $l$, while the 
$u_{2l}$ may be arbitrary.
If $n$ is even, then $u_i\leq 1$. If $u_i=0$ then also $u_{i+1}=0$.
If $n=2$, we see no difference between $E_M$ and $E_R$, and indeed
the sequences $(m_M^{(s)})$ and $(m_R^{(s)})$ are of the same shape.

\subsection{$D_{2k+1}^{k+1,2}$}\label{dtkpe}
The configuration is connected to vertices $E_R$ and $E_L$. We claim
that \textit{the coefficient $z_L$ of $E_L$ in the fundamental 
cycle can be at most
two}. We compute the fundamental cycle with $E_L$ as central
vertex. The relevant information on the cycle $Y^{(1)}_{\Delta,R}$ 
is given in the entry for $RD_{2k+1}^{k+1,2}$ in Table \ref{tableBB}.
If the coefficient of $E_R$ is $s$, then the 
multiplicity of the vertex adjacent to $E_L$ is $m_L^{(1)}
  =\lfloor\frac{2k+1+s}2\rfloor$.
We assume that $E_L\cdot Z^{(1)}=1$.
If $s=2t+1$, then $Y^{(2)}_{\Delta,R}=\emptyset$ 
and the computation stops with $z_L=2$ and $z_R=2t+1$.
If $s=2t+2$, then $\Gamma^{(2)}_{\Delta,R}$ has only an
$A_{2k}^{1,1}$-configuration
between $E_L$ and $E_R$, so  $E_R$ is not a bad vertex for
$Y^{(2)}_{\Delta,R}$ and $m_L^{(2)}=1$. As $\lfloor\frac{2k+1+s}2\rfloor=
k+1+t\geq 3$, the computation again stops with $z_L=2$.
Depending on whether $E_R\cdot  Z^{(1)}=0$ or less, $z_R= 2t+3$
or $z_R=2t+2$.

\subsection{$D_{2k}^{k,2}$}
In  this case only one of the vertices
$E_L$ and $E_R$ is bad. In the symmetric case $k=2$ we assume that
$E_R$ is the bad vertex. We compute as in the previous case
with $E_L$ as central vertex. 
If the coefficient of $E_R$ in  $Y^{(1)}_{\Delta,R}$ is $s$
(with $s>1$ if and only if $E_R$ is bad), then  $m_L^{(1)}
  =\lfloor\frac{2k+s}2\rfloor$.
If $s=2t-1$, then $m_L^{(2)}=1$.  If $k\geq 3$, then  $\lfloor\frac{2k+s}2\rfloor=
k+t-1\geq k\geq3$. For $k=2$ we assumed $s>1$, so $t>1$
and again $k+t-1\geq 3$. So the computation stops with $z_L=2$,
and $z_R= 2t-1$ or $z_R=2t$.
If $s=2t$, then $Y^{(2)}_{\Delta,R}=\emptyset$ 
and the computation stops with $z_L=2$ and $z_R=2t$.

\subsection{$A_n^{2,k,2}$}
As in the cases ${}^{\it I}\!A_n^{2,k}$ and ${}^{\it II}\!A_n^{k,2}$
we compute with the rupture point in the $A_n^{2,k,2}$-configuration as central 
vertex $E_0$.
The sequence $(m_L^{(s)})$
 is
 \[
(1,(1^{k-2},0)^{s_1},1,(1^{k-2},0)^{s_2},1,(1^{k-2},0)^{s_3},\dots)\;,
\] 
the sequence 
$(m_M^{(s)})$ is
\[
(1,0^{k+t_1-2},1,0^{k+t_2-2},1,0^{k+t_3-2},\dots)\;.
\]
and finally $(m_R^{(s)})$ is
 \[
(1,(1^{n-k+1},0)^{u_1},1,(1^{n-k+1},0)^{u_2},1,(1^{n-k+1},0)^{u_3},\dots)\;.
\]
First suppose $E_M$ is a bad vertex, i.e., $t_1=0$. We may
assume that $k>2$.
Then $E_L$ is not a bad vertex, $s_1>0$, except possibly  if
$n$ has the lowest possible value $2k-3$,  when there is an arrowhead
between $E_M$ and $E_L$ at $E_0$. In that case the chains from
$E_0$ to $E_L$ and $E_R$ are equally long. As $n-k+1=k-2$,
not all three of
$s_1$, $t_1$ and $u_1$ are zero, so upon relabeling we may assume 
also here that $s_1>0$. As in the case 
${}^{\it I}\!A_n^{2,k}$ we find that the computation stops with the
first $0$ in the sequence $(m_R^{(s)})$, or in the step immediately
after. It is however possible that there is no $0$ in this sequence;
this happens if $u_i=0$ for all $i$.

If $t_1>0$, then $s_1=0$, and if $n=2k-3$, also $u_1=0$.
For most values  of $s$ we will have $m_L^{(s)}+ m_R^{(s)}=2$,
but we want that  $m_L^{(s)}+ m_R^{(s)}=1$
for $s=p(k-1)+\sum_{i=1}^p  t_i +1$ for all $p\geq1$.
We determine on  which places in the sequence
$(m_L^{(s)})$ there are zeroes. Let $\sum_{j=1}^{i-1}s_j< r
\leq \sum_{j=1}^{i}s_j$. Then the $r$-th zero is on place
$r(k-1)+i$. Similarly the   $r$-th zero  in the sequence
$(m_R^{(s)})$
is on place $r(n-k+2)+i$, if $\sum_{j=1}^{i-1}u_j< r
\leq \sum_{j=1}^{i} u_j$.

If $k=2$, we may upon relabeling assume that $t_1>0$.
Then the same description holds.
In particular, if $n=1$, we have the sequences
$(1,0^{s_1},1,0^{s_2}, \dots)$, $(1,0^{t_1},1,0^{t_2}, \dots)$
and $(1,0^{u_1},1,0^{u_2}, \dots)$.  Once again we stress that
we do not investigate, which values of $s_i$, $t_i$ and $u_i$
are possible.

\subsection{Multiplicity  at most two}\label{twocurvestwo}
In the previous subsections we have tried to make the
multiplicity of the fundamental cycle at non-$(-2)$'s
as large as possible. The computations above also tell us
when the multiplicity does not exceed two. 
Now we make
the conditions explicit in terms of the 
multiplicities of the other components of the graph, attached
to the two non-$(-2)$'s.
Let $E_a$ be one of these vertices. Then as before $\Gamma_a$
is the union of connected components, attached to $E_a$.
Let $E_a+Y^{(1)}_a$ be the
fundamental cycle on $\{E_a\}\cup \Gamma_a$,  
and denote by $n^{(1)}_a$  the sum of the
multiplicities of $Y^{(1)}_a$ at the vertices of 
$\Gamma_a$, adjacent to $E_a$.
At the stage of the computation of the fundamental cycle, when
the multiplicity of $E_a$ has increased to 2, we need the fundamental
cycle $E_a+Y^{(2)}_a$ on $\{E_a\}\cup \Gamma_a^{(2)}$, where
$\Gamma_a^{(2)}$ is a connected component with vertices
satisfying $E_i\cdot (E_a+Y^{(1)}_a)=0$; then $E^{(2)}_a$  is
the sum of the multiplicities of $Y^{(2)}_a$, adjacent to $E_a$.

\subsubsection{$A_n^{1,1}$}
As before we assume that $E_L$ is the bad vertex. The computation
with $E_L$ as central vertex should stop at $s=2$, so
$E_L\cdot Z^{(1)}=1$ and $E_L\cdot Z^{(2)}\leq0$.
As the multiplicity of $E_R$ also should be two, we need
$E_R\cdot Z^{(1)}=0$.
This gives us 
\[
\begin{aligned}[t]
n_L^{(1)}&=b_L,\\n_L^{(2)}&\leq b_L-2,
\end{aligned}\qquad
\begin{aligned}[t]
n_R^{(1)}&=b_R-1,\\
n_R^{(2)}&\leq b_R-1.
\end{aligned}
\] 
\subsubsection{${}^{\it I}\!A_n^{2,k}$}
First consider the case that $E_M$ is the bad vertex, so $t_1=0$
and $s_1>0$. If $s_1>1$, then the computation stops before 
the multiplicity $z_L$ becomes two, or  $z_M$ becomes at least three.
Therefore $s_1=1$. We have the sequences
$(1,1^{k-2},0,1,1^{k-2},0,\dots)$ and $(1,0^{k-2},1,0^{k+t_2-2},1,\dots)$.
As $n-k+1\geq k$  we have $n-k+2\geq k+1$.  The condition
that the multiplicities do not exceed two depend on $n$.
If $n-k+2\leq 2k-2$, the computation always stops at $s=n-k+2$.
If $n-k+2=2k-1$, then we need $t_2\geq 1$ and if 
$n-k+2\geq 2k$, then we  $t_2\geq 2$. 
Thus  
\[
\begin{aligned}[t]
n_L^{(1)}&=b_L-2,\\n_L^{(2)}&\leq b_L-2,
\end{aligned}\qquad
\begin{aligned}[t]
n_M^{(1)}&=b_M-k+1,\\
n_M^{(2)}&\leq 
\begin{cases}b_M-k+1,& \text{if } n\leq 3k-4,\\
b_M-k,& \text{if }  n= 3k-3,\\
b_M-k-1,& \text{if } n\geq 3k-2.
\end{cases}
\end{aligned}
\] 
If $E_L$ is the bad vertex, we have $s_1=0$ and we need
$t_1=1$. Furthermore $s_2\geq1$. If $n-k+2=k+1$, the 
computation stops at $s=n-k+2$. Otherwise we need $s_2>1$.
This gives 
\[
\begin{aligned}[t]
n_L^{(1)}&=b_L-1,\\n_L^{(2)}&\leq
\begin{cases}b_L-2,& \text{if } n=2k-1,\\
b_L-3,& \text{if }  n\geq 2k,
\end{cases}
\end{aligned}\qquad
\begin{aligned}[t]
n_R^{(1)}&=b_R-k,\\
n_R^{(2)}&\leq b_R-k.
\end{aligned}
\]

\subsubsection{${}^{\it II}\!A_n^{k,2}$}
If  $u_1>0$, so $t_1=0$,  the computation stops too early
or the coefficient of $E_M$ becomes too high.
We need $u_2>0$. The value of $t_2$ depends
again on $n$. The results also hold for $k=2$.
\[
\begin{aligned}[t]
n_M^{(1)}&=b_M-k+1,\\
n_M^{(2)}&\leq 
\begin{cases}b_M-k+1,& \text{if } n\leq 3k-5,\\
b_M-k,& \text{if }  n= 3k-4,\\
b_M-k-1,& \text{if } n\geq 3k-3,
\end{cases}
\end{aligned}
\qquad
\begin{aligned}[t]
n_R^{(1)}&=b_R-1,\\n_R^{(2)}&\leq b_R-2.
\end{aligned}
\] 

\subsubsection{$D_{2k+1}^{k+1,2}$}
In the notation of \ref{dtkpe} we need that $s=2$ and $E_R\cdot Z^{(1)}<0$.
This gives us
\[
\begin{aligned}
n_L^{(1)}&=b_L-k,\\n_L^{(2)}&\leq b_L-k,
\end{aligned}\qquad
\begin{aligned}
n_R^{(1)}&=b_R-1,\\
n_R^{(2)}&\leq b_R-3.
\end{aligned}
\] 

\subsubsection{$D_{2k}^{k,2}$}
In this case $s\leq 2$ and $z_R=2$. We first assume $k>2$.
This gives two possibilities. If $s=1$ we obtain
\[
\begin{aligned}
n_L^{(1)}&=b_L-k+1,\\n_L^{(2)}&\leq b_L-k+1,
\end{aligned}\qquad
\begin{aligned}
n_R^{(1)}&=b_R-2,\\
n_R^{(2)}&\leq b_R-2
\end{aligned}
\] 
and for $s=2$
\[
\begin{aligned}
n_L^{(1)}&=b_L-k,\\n_L^{(2)}&\leq b_L-k,
\end{aligned}\qquad
\begin{aligned}
n_R^{(1)}&=b_R-1,\\
n_R^{(2)}&\leq b_R-2.
\end{aligned}
\] 
The last formula also works for the symmetric case $k=2$, if we assume
that $E_R$ is the bad vertex.

\subsubsection{$A_n^{2,k,2}$} \label{antkt}
We have to determine the conditions that at least
two multiplicities become 2, whereas none may become 3.
We argue as  in the cases ${}^{\it I}\!A_n^{2,k}$ 
and ${}^{\it II}\!A_n^{k,2}$. If $E_R$ ($t_1=0$) is bad we may
assume that $k>2$. If $s_1>1$, the multiplicity
of $E_L$ remains 1, which is seen by the absence of
the entry  for $n_L^{(2)}$:
\[
\begin{aligned}[t]
n_L^{(1)}&\leq b_L-3,\\
\end{aligned}
\qquad
\begin{aligned}[t]
n_M^{(1)}&=b_M-k+1,\\
n_M^{(2)}&\leq 
\begin{cases}b_M-k+1,& \text{if } n\leq 3k-6,\\
b_M-k,& \text{if }  n\geq 3k-5,
\end{cases}
\end{aligned}
\qquad
\begin{aligned}[t]
n_R^{(1)}&=b_R-1,\\n_R^{(2)}&\leq b_R-2.
\end{aligned}
\] 
If $s_1=1$ and $u_1<0$ (so $n_R^{(1)}<b_R-1$), then $n> 2k-3$;
for $n=2k-3$ one has, if necessary, to interchange
$E_L$ and $E_R$. We get
\[
\begin{aligned}[t]
n_L^{(1)}&= b_L-2,\\ n_L^{(2)}&\leq b_L-2,
\end{aligned}
\qquad
\begin{aligned}[t]
n_M^{(1)}&=b_M-k+1,\\
m_M^{(2)}&\leq 
\begin{cases}b_M-k+1,& \text{if } n\leq 3k-5,\\
b_M-k,& \text{if }  n= 3k-4, \\
b_M-k-1,& \text{if }  n\geq 3k-3,
\end{cases}
\end{aligned}
\qquad
\begin{aligned}[t]
n_R^{(1)}&\leq b_R-2.\\
\end{aligned}
\] 
It is also possible that all three multiplicities are 2:
\[
\begin{aligned}[t]
n_L^{(1)}&= b_L-2,\\ n_L^{(2)}&\leq b_L-2,
\end{aligned}
\qquad
\begin{aligned}[t]
n_M^{(1)}&=b_M-k+1,\\
n_M^{(2)}&\leq 
\begin{cases}b_M-k+1,& \text{if } n\leq 3k-6,\\
b_M-k,& \text{if }  n= 3k-5,\\
b_M-k-1,& \text{if }  n\geq 3k-4,
\end{cases}
\end{aligned}
\qquad
\begin{aligned}[t]
n_R^{(1)}&=b_R-1.\\ n_R^{(2)}&\leq b_R-2.
\end{aligned}
\] 
If $E_M$ is not bad, we allow that $k=2$.
\[
\begin{aligned}[t]
n_L^{(1)}&= b_L-1,\\ n_L^{(2)}& \leq b_L-2
\end{aligned}
\qquad
\begin{aligned}[t]
n_M^{(1)}&=b_M-k-1,\\
\end{aligned}
\qquad
\begin{aligned}[t]
n_R^{(1)}&=b_R-1.\\ n_R^{(2)}&\leq b_R-2.
\end{aligned}
\] 
Also now it is possible that all three multiplicities are 2:
\[
\begin{aligned}[t]
n_L^{(1)}&= b_L-1,\\ 
n_L^{(2)}& \leq 
 \begin{cases}B_L-2,& \text{if } n=2k-3,\\
     b_L-3,& \text{if }  n\geq 2k-2,
\end{cases}
\end{aligned}
\qquad
\begin{aligned}[t]
n_M^{(1)}&=b_M-k,\\
n_M^{(2)}&\leq b_M-k, 
\end{aligned}
\qquad
\begin{aligned}[t]
n_R^{(1)}&=b_R-1.\\ n_R^{(2)}&\leq b_R-2.
\end{aligned}
\] 

\section{Low degree}
The classification of rational graphs of 
degree three was given by Artin \cite{ar},
degree four by the author \cite{st-p} and
degree five by Tosun et al.~\cite{tosun}.
In these cases there is at most one  non-reduced non-$(-2)$,
so the classification can be written 
using the results of Sections 
\ref{arfc} and \ref{onrnmt}.
For degree six one new case arises,
with two non-reduced   non-$(-2)$'s; here the results of Subsection 
\ref{twocurvestwo} suffice, as we presently shall make explicit.
For degree 7 one can use the same methods; we do not go into detail.
Things become more complicated for degree 8, where possibility
of three  non-reduced   non-$(-2)$'s appears. We classify the occurring
graphs in this section. 

\subsection{Degree six}
We start with the classification of graphs of canonical models.
The ones with reduced fundamental cycle are given in Table \ref{uptillsix}.
From it one can also infer the other possibilities: just replace some
vertices with weight $-b$ with a vertex of weight $-3$ and multiplicity
$b-2$, or the $(-6)$ by a $(-4)$ of multiplicity 2. 
We do not treat all cases, where there is only one non-$(-2)$
with higher multiplicity, but  
We give
partial results for some cases and as example  we list
the complete classification in the case of highest
multiplicity. 

The new case in degree 6 is that  there are two $(-3)$-vertices 
with multiplicity two in the fundamental cycle.
The possible configurations are described in Section
\ref{twocurvestwo}. We have to specialise to the case that the vertex weights 
are 3.

We write $C(m_1,m_2)$ for any combination of RDP-configurations
realising the multiplicity sequence $(m_1,m_2)$, and $C(m_1,\leq m_2)$
for configurations where the total second multiplicity is at most $m_2$.
The notation $C(0,0)$ stands for the empty configuration.
These combinations can be found from Table \ref{tableAA}; e.g.,
$(3, \leq 1)$ stands for $2A_1^1+A_n^1$ ($n\geq1$), $A_5^3$, $A_6^3$, 
${}^{\it I}\!D_k^2+A_1^1$,
$A_3^2+A_1^1$, $A_4^2+A_1^1$, ${}^{\it II}\!D_5^2+A_1^1$,
${}^{\it II}\!D_6^3$, ${}^{\it II}\!D_7^3$ and $E_7^3$.

\begin{prop}
Suppose the graph of the RDP-resolution consist of two $(-3)$-vertices,
both with multiplicity 2.
Then they are connected by one of the  RDP-configurations, 
listed in Table \ref{tablezes} together with the
the other configurations at the left and the right vertex.
\end{prop}
\begin{table} \caption{} \label{tablezes}
\renewcommand{\arraystretch}{1.2}
$
\begin{array}{lll}
\mbox{name}& \mbox{left}& \mbox{right}
\\ 
\hline
\\[-1.2\bigskipamount]
A_n^{1,1}    &  C( 3, \leq1) &C(2, \leq 2) \\
{}^{\it I}\!A_3^{2,2}  &  C(2, \leq1) & C(1, \leq 1) \\
{}^{\it I}\!A_{\geq 4}^{2,2},   &  C(2, 0) & C(1, \leq 1) \\
{}^{\it I}\!A_5^{2,3}  &  C(2, \leq1) & C(0, 0) \\
{}^{\it I}\!A_{\geq 6}^{2,3}  &  C(2, 0) & C(0, 0) \\
{}^{\it I}\!A_5^{2,3}  &  C(1, \leq1) & (1, \leq 1) \\
{}^{\it I}\!A_6^{2,3}  &  C(1, \leq1) & C(1,  0)  \\
{}^{\it I}\!A_7^{2,4}  &  C(1, \leq1) & C(0, 0) \\
{}^{\it I}\!A_8^{2,4}  &  C(1, \leq1) & C(0, 0) \\
{}^{\it II}\!A_2^{2,2}  &  C(2, \leq1) & C(2, \leq 1) \\
{}^{\it II}\!A_{\geq3}^{2,2}  &  C(2,0) & C(2, \leq 1) 
\\ 
\hline
\end{array}{}\qquad
\begin{array}{lll}
\mbox{name}& \mbox{left}& \mbox{right}
\\ 
\hline
\\[-1.2\bigskipamount]
{}^{\it II}\!A_4^{3,2}  &  C(1, \leq1) & C(2, \leq 1) \\
{}^{\it II}\!A_5^{3,2}  &  C(1, 0) & C(2, \leq 1) \\
{}^{\it II}\!A_6^{4,2}  &  C(0,0) & C(2, \leq 1) \\
{}^{\it II}\!A_7^{4,2}  &  C(0, 0) & C(2, \leq 1) \\
D_{4}^{2,2}  & C(1 , \leq 1) & C(2 ,\leq 1) \\
D_{5}^{3,2}  & C(1 , \leq 1) & C(2 ,0) \\
D_{6}^{3,2}  & C(0,0) & C(2 ,\leq 1) \\
D_{6}^{3,2}  & C(1 , \leq 1) & C(1 ,\leq 1) \\
D_{7}^{4,2}  & C(0 , 0)  & C(2 ,0) \\
D_{8}^{4,2}  & C(0,0)   & C(1 , \leq 1) 
\\ 
\hline
\end{array}
$
\end{table}

\begin{prop}
Suppose the graph of the RDP-resolution consist of one $(-3)$-vertex,
with multiplicity 4.
The following combinations of RDP-configurations are possible.
\[
\begin{array}{lll}
A_4^2+2A^1_3,& 
A_4^2+A^2_7,&
A_1^1+A_6^2+A^1_{\geq3},
\\ 
{}^{\it II}\!D_5^2+A_6^2,&
A_1^1+A_2^1+A_8^2,&
A_1^1+A_2^1+A_3^1+A^1_{\geq3},
\\
{}^{\it II}\!D_5^2+A_2^1+A^1_{\geq3}, &
A_1^1+A_{10}^3, &
A_1^1+A_2^1+A_7^2. 
\end{array} 
\]
\end{prop}

\begin{proof}
We argue as  in the proof of Proposition \ref{propsix}.
We first consider only RDP-configurations of type $A_n^1$ and
$A_{2k}^2$. 
We need  either $A_1^1$ or $A_4^2$. 
As $A_4^2$ gives the sequence
$(2,1,1,2,0)$, we need the sequence $(2,2,2,0)$, so
the configuration $2A^1_3$.  If there is exactly one $A_1^1$, we
further need  $A_2^1$ or $A_6^2$. In the first case we can complete
with $A_8^2$ or $A_3^1+A^1_n$ with $n\geq3$, in the second only with
$A^1_n$, $n\geq3$.
Table \ref{tableAA} gives the possible equivalent configurations.
\end{proof}

Next we consider the case that the hypertree for the RDP-resolution
has a $T$-joint. The smallest tree realising it looks as follows.
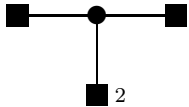
\[
\bp(2,1.2)(0,-1) \put(0,0){\vir}\put(0,0){\lijn}\put(1,0){\cir}
\put(1,0){\lijn}\put(2,0){\vir}
\put(1,0){\line(0,-1){1}} \put(1,-1){\vit{\;2}}\ep
\]
As drawn, the vertex $E_M$ has multiplicity two. The other cases are
also possible, and occur in the classification, but they give
basically the same graph.

\begin{prop}
If the hypertree of the RDP-resolution has a $T$-joint and
$E_M$ is the vertex of higher multiplicity, the configurations
$MA_3^{2,3,2}+C(1,\leq1 )$, $MA_4^{2,3,2}+C(1,\leq1 )$,
$MA_{\geq 5}^{2,3,2}+C(1,0 )$, $MA_5^{2,4,2}$, $MA_6^{2,4,2}$ and
$MA_{\geq 7}^{2,4,2}$ can be attached to $E_M$; at $E_R$  an
$A_n^1$ is possible and also at $E_L$ in the symmetric case
of minimal $n=2k-3$.
To $E_R$ of  higher multiplicity the configurations
$RA_n^{2,2,2}+C(2,\leq2 )$ and $RA_n^{2,3,2}+C(2,\leq1 )$ 
can be attached; at $E_L$ an $A_n^1$ is possible and also at $E_M$ in 
the  case $k=2$.
The last possibility is $LA_{\geq2}^{2,2,2}+C(2,\leq1 )$ with an
optional $A_n^1$ at  $E_R$.
\end{prop}

\begin{proof}
We use Table \ref{tableCC}. The only thing to note is that we stop
the computation earlier, at step two, so in the case $LA_{\geq3}^{2,2,2}$
the multiplicity at $E_M$ does not reach the value $n+1$, but remains 3.
\end{proof}

For two other cases, with the following graphs for the 
RDP-resolution,
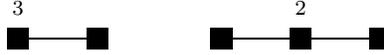
\[
\bp(1,.7) \put(0,0){\vi{3}}\put(0,0){\lijn}\put(1,0){\vir}\ep
\qquad\qquad
\bp(2,.7) \put(0,0){\vir}\put(0,0){\lijn}\put(1,0){\vi{2}}
\put(1,0){\lijn}\put(2,0){\vir}\ep
\]
we only show how they can be realised, using 
configurations of type  $LA_n^{1,1}$ for the connection
to  other $(-3)$'s, and configurations
$C(m_1,\leq m_2)$ and $C(m_1,m_2,\leq m_3)$; as before this notation stands
for any  combination of configurations, realising a multiplicity
sequence. 

\begin{prop}
Suppose the graph of the RDP-resolution consist of two $(-3)$-vertices,
one with multiplicity 3.
This type can be realised by attaching to the curve of multiplicity 3
a combination $LA_0^{1,1}+C(3,3,\leq2)$,  $LA_1^{1,1}+C(3,2,\leq2)$
or  $LA_{\geq2}^{1,1}+C(3,2,\leq1)$.

Two reduced $(-3)$'s with a $(-3)$ of multiplicity 2
in between can be realised by attaching to the vertex of multiplicity 2
$LA_0^{1,1}+ LA_0^{1,1}+C(2,\leq2)$, $LA_{\geq1}^{1,1}+ LA_0^{1,1}+C(2,\leq1)$
or $LA_{\geq1}^{1,1}+ LA_{\geq1}^{1,1}+C(2,0)$.
\end{prop}

The three remaining cases are easier.

\subsection{Degree eight}
We consider here only the cases that there are three $(-3)$-vertices,
all with multiplicity two in the fundamental cycle. Either all three
are connected by a single $A_n^{2,k,2}$ configuration, or they form a chain.
The first possibility is a special case of Section \ref{antkt}.

\begin{prop}
Suppose the graph of the RDP-resolution consist of three $(-3)$-vertices,
all with multiplicity 2, connected by a single $A_n^{2,k,2}$ configuration.
Then following values for $n$ and $k$ are possible, with the given 
other configurations at each vertex.
\[\renewcommand{\arraystretch}{1.2}
\begin{array}{llll}
\mbox{\rm name}& \mbox{\rm left}& \mbox{\rm middle} & \mbox{\rm right}
\\
\hline
\\[-1.2\bigskipamount]
A_1^{2,2,2}    &   C(2, \leq 1) &C(1, \leq 1) & C(2, \leq 1) \\
A_n^{2,2,2}    &   C(2, 0) &C(1, \leq 1) & C(2, \leq 1) \\
A_3^{2,3,2}    &   C(2, \leq 1) & C(0,0) & C(2, \leq 1) \\
A_n^{2,3,2}    &   C(2, 0) &  C(0,0)  & C(2, \leq 1) \\
A_3^{2,3,2}    &   C(1, \leq 1) &  C(1, \leq 1) & C(2, \leq 1) \\
A_n^{2,3,2}    &    C(1, \leq 1) &  C(1,0)  & C(2, \leq 1) \\
A_5^{2,4,2}    &   C(1, \leq 1) &  C(0,0) & C(2, \leq 1) \\
A_6^{2,4,2}    &    C(1, \leq 1) &  C(0,0)  & C(2, \leq 1)\\
\hline
\end{array}
\]
\end{prop}

Finally we consider a chain of non-reduced $(-3)$'s. Let the 
vertices be called $E_L$, $E_M$ and $E_R$. We compute the 
fundamental cycle as described  in Section \ref{compfc} with
$E_M$ as central vertex. The complement $\Gamma\setminus \{E_M\}$
decomposes into the connected components $\Gamma_L$ and $\Gamma_R$, 
containing respectively $E_L$ and  $E_R$, and the union $\Gamma_M$ 
of the remaining components. We consider the multiplicity sequences
$(m_L^{(s)})=(m_L^{(1)},m_L^{(2)})$, 
$(m_M^{(1)},m_M^{(2)})$ and
$(m_R^{(1)},m_R^{(2)})$. We need that $m_L^{(1)}+m_M^{(1)}+m_R^{(1)}=4$
and $m_L^{(2)}+m_M^{(2)}+m_R^{(2)}\leq 2$. Upon interchanging
$E_L$ and $E_R$ we may assume that $m_L^{(1)}\geq m_R^{(1)}$.

\begin{prop}
For a chain of three $(-3)$'s with multiplicity $2$ in the 
fundamental cycle the following multiplicity sequences are possible,
when computing with the middle vertex as central vertex.
\[
\begin{array}{lll}
(m_L^{(s)}) & (m_M^{(s)}) &
(m_R^{(s)})
\\ [2mm]
\hline
\\[-\bigskipamount]
(3,\leq 1) & (0,0) & (1,1)\\
(2,\leq2) & (0,0) & (2,0)\\
(2,1) &(0,0) &(2,1)\\[2mm]
\hline
\end{array}
\qquad
\begin{array}{lll}
(m_L^{(s)}) & (m_M^{(s)}) &
(m_R^{(s)})
\\ [2mm]
\hline
\\[-\bigskipamount]
(2,0) & (1,\leq 1) & (1,1)\\
(2,1) & (1,0) & (1,1)\\
(1,1) &(2,0) &(1,1)\\[2mm]
\hline
\end{array}
\]
\end{prop}

The configurations giving the required values for $(m_M^{(1)},m_M^{(2)})$
can be read off from Table \ref{tableAA}.
We have $C(1,0)=A_1^1$, $C(1,1)=A_n^1$, $n>1$, and $C(2,0)$ can be
$2A_1^1$, $A_3^2$ or ${}^{\it I}\!D_k^2$. For 
$(m_L^{(1)},m_L^{(2)})$ and
$(m_R^{(1)},m_R^{(2)})$ we use Table \ref{tableBB}. It suffices
to describe the possible configurations for $E_L$. The result is given 
in Table \ref{tableight}.

We have to distinguish  cases depending on whether $E_L$ is a bad vertex for 
$\Gamma_L\cup\{E_M\}$ or not. If bad, then the multiplicity of $E_L$ in
$Y_L^{(1)}$ is two, and the multiplicity does not increase in the second step.
This means that $E_i\cdot Z^{(1)}<0$ for some vertex $E_i$ on the chain
between $E_L$ and $E_M$.  This is an extra condition, which excludes
a number of cases from  Table \ref{tableBB}. If $E_L$ is not bad,
then its multiplicity in $Y_L^{(1)}$ is one, and $E_i\cdot Z^{(1)}=0$ for 
all vertices $E_i$ on the chain between $E_L$ and $E_M$, including $E_L$.
In this case $m_L^{(2)})\geq1$.

\begin{table}[h]\caption{}
\label{tableight}
\renewcommand{\arraystretch}{1.2}
$
\begin{array}{lll}
(m_L^{(1)},m_L^{(2)}) & \;\;E_L\mbox{ bad} &
\;\;E_L\mbox{ not bad}
\\ 
\hline
\hline
\\[-1.1\bigskipamount]
(1,1)  &  &  LA_n^{1,1} + C(2,\leq 2) 
\\
\hline
\\[-1.1\bigskipamount]
(2,0) & LA_n^{1,1} + C(3,\leq 1) \\ 
\hline
\\[-1.1\bigskipamount]
(2,1) & M^{\it II}\!A_{2}^{2,2} +  C(2,\leq 1) 
                  & L^{\it I}\!A_{3}^{2,2} +  C(1,\leq 1) \\
    & M^{\it II}\!A_{3}^{2,2} +  C(2,0) 
                  & LD_{4}^{2,2} +  C(1,\leq 1) \\
     & M^{\it II}\!A_{4}^{3,2} +  C(1,\leq 1) 
                  & M^{\it I}\!A_{5}^{2,3} +  C(0,0) \\
    & M^{\it II}\!A_{5}^{3,2} +  C(1,0) 
                  & LD_{6}^{3,2} +  C(0,0) \\
     & M^{\it II}\!A_{6}^{4,2} +  C(0,0)  \\
   & M^{\it II}\!A_{7}^{4,2} +  C(0,0)  \\ 
\hline
\\[-1.1\bigskipamount]
(2,2) & R^{\it II}\!A_{\geq3}^{2,2} +  C(2,\leq 2) 
                  & L^{\it I}\!A_{\geq4}^{2,2} +  C(1,\leq 1) \\
     & LD_{5}^{2,2} +  C(1,\leq 1) 
                  & M^{\it I}\!A_{\geq6}^{2,3} +  C(0,0) \\
                  & LD_{7}^{3,2} +  C(0,0) \\ 
\hline
\\[-1.1\bigskipamount]
(3,0)  & L^{\it I}\!A_{3}^{2,2} +  C(2,0) \\
     & LD_4^{2,2} + C(2,\leq1) \\
    & M^{\it I}\!A_{5}^{2,3} +  C(1,0) \\
    & LD_6^{3,2} + C(1,\leq1) \\
    & M^{\it I}\!A_{7}^{2,4} +  C(0,0) \\
     & LD_8^{4,2} + C(0,0) \\ 
\hline
\\[-1.1\bigskipamount]
(3,1) &   L^{\it I}\!A_{\geq4}^{2,2} +  C(2,0) 
    & L^{\it I}\!A_{5}^{2,3} +  C(1,\leq 1)\\
      &R^{\it II}\!A_{4}^{3,2} +  C(2,0) & RD_6^{3,2} + C(1,\leq1)\\
     & RD_5^{3,2} + C(2,0) \\
    & M^{\it I}\!A_{6}^{2,3} +  C(1,0) \\
    & M^{\it I}\!A_{8}^{2,4} +  C(0,0) 
    \\
\hline
\end{array}
$
\end{table}
We give the graphs for the simplest ways to realise 
a chain of three $(-3)$'s with multiplicity $2$, depending on $E_L$ or 
$E_R$ being bad. Again it would be interesting to know whether these
singularities have the same format.
\[
\unitlength=25pt
\bp(4.2,2)(-.1,-1)
\put(0,0){\cir} \put(0,0){\lijn}
\put(1,0){\vir} \put(1,0){\lijn}
\put(1,0){\line(0,1){1}} \put(1,1){\cir}
\put(1,0){\line(0,-1){1}} \put(1,-1){\cir}
\put(2,0){\vir} \put(2,0){\lijn}
\put(3,0){\vir} \put(3,0){\lijn}
\put(3,0){\line(0,1){1}} \put(3,1){\cir}
\put(3,0){\line(0,-1){1}} \put(3,-1){\cir}
\put(4,0){\cir} 
\ep
\qquad
\bp(4.2,2)(-.1,-1)
\put(0,0){\cir} \put(0,0){\lijn}
\put(1,0){\vir} \put(1,0){\lijn}
\put(1,0){\line(0,1){1}} \put(1,1){\cir}
\put(1,0){\line(0,-1){1}} \put(1,-1){\cir}
\put(2,0){\vir} \put(2,0){\lijn}
\put(3,0){\vir} \put(3,0){\lijn}
\put(2,0){\line(0,-1){1}} \put(2,-1){\cir}
\put(3,0){\line(0,-1){1}} \put(3,-1){\cir}
\put(4,0){\cir} 
\ep
\qquad
\bp(4.2,2)(-.1,-1)
\put(0,0){\cir} \put(0,0){\lijn}
\put(1,0){\vir} \put(1,0){\lijn}
\put(2,0){\line(0,1){1}} \put(2,1){\cir}
\put(1,0){\line(0,-1){1}} \put(1,-1){\cir}
\put(2,0){\vir} \put(2,0){\lijn}
\put(3,0){\vir} \put(3,0){\lijn}
\put(2,0){\line(0,-1){1}} \put(2,-1){\cir}
\put(3,0){\line(0,-1){1}} \put(3,-1){\cir}
\put(4,0){\cir} 
\ep
\]

\end{document}